\documentclass[a4paper,11pt,letterpaper, english]{amsart}
\title[Rigidity of 3-dim metric spaces]{Rigidity of actions on metric spaces close to three dimensional manifolds}

\author{No\'{e} B\'{a}rcenas }
\author{Manuel Sedano-Mendoza}
                \email{barcenas@matmor.unam.mx \\ msedano@matmor.unam.mx}
 \address{Centro de Ciencias Matem\'aticas. UNAM \\Ap.Postal 61-3 Xangari. Morelia, Michoac\'an MEXICO 58089}
 \date{\today}

\usepackage[colorinlistoftodos]{todonotes}

\usepackage{hyperref}
\usepackage{pdfsync}
\usepackage{calc}
\usepackage{enumerate,amssymb}
\usepackage[arrow,curve,matrix,tips,2cell]{xy}
  \SelectTips{eu}{10} \UseTips
  \UseAllTwocells

\usepackage[colorinlistoftodos]{todonotes}

\newtheorem{theorem}{Theorem}[section]
\newtheorem{lemma}[theorem]{Lemma}
\newtheorem{proposition}[theorem]{Proposition}
\newtheorem{corollary}[theorem]{Corollary}
\newtheorem{conjecture}[theorem]{Conjecture}

\newtheorem{definition}[theorem]{Definition}
\newtheorem{example}[theorem]{Example}

\newtheorem{problem}[theorem]{Problem}
\newtheorem{remark}[theorem]{Remark}

{\catcode`@=11\global\let\c@equation=\c@theorem}

\newcommand{\IC}{{\mathbb C}}

\newcommand{\IH}{{\mathbb H}}

\newcommand{\IN}{{\mathbb N}}

\newcommand{\IR}{{\mathbb R}}
\newcommand{\IS}{{\mathbb S}}

\newcommand{\IZ}{{\mathbb Z}}

\newcommand{\higherlim}[3]{{\setbox1=\hbox{\rm lim}
        \setbox2=\hbox to \wd1{\leftarrowfill} \ht2=0pt \dp2=-1pt
        \mathop{\vtop{\baselineskip=5pt\box1\box2}}
        _{#1}}^{#2}#3}

\newcommand{\version}[1]                       %marks the date of last editing and compilation
{\begin{center} last edited on #1\\
last compiled on \today\\
name of texfile: \jobname
\end{center}  

\typeout{----------------------------  linluesau.tex  ----------------------------}

}

\newcounter{commentcounter}

\newcommand{\g}{\mathfrak{g}}
\newcommand{\R}{\mathbb{R}}
\newcommand{\C}{\mathbb{C}}
\newcommand{\Z}{\mathbb{Z}}
\renewcommand{\H}{\mathbb{H}}

\begin{document}

\maketitle

\begin{abstract}
In  this  paper we propose a metric variation on the $C^0$-version of the Zimmer program for three manifolds. 
After  a  reexamination of  the  isometry  groups  of  geometric three-manifolds,  we  consider  homomorphisms  defined  on  higher  rank  lattices to  them and  establish a  dichotomy betweeen  finite  image or  infinite  volume of the  quotient. 
Along  the  way,  we  enumerate classification  results  for  actions  of  finite  groups  on three manifolds where  available,  and  we  give an  answer   to  a  metric  variation  on topological  versions  of  the  Zimmer  program  for  aspherical  three-manifolds, as asked  by Weinberger  and  Ye,  which are  based  on  the  dichotomy stablished  in  this  work and  known  topological  rigidity phenomena  for  three  manifolds. 
Using results  by  John Pardon  and  Galaz-Garc\'ia-Guijarro, the dichotomy  for  homomorphisms  of  higher  rank  lattices  to  isometry  groups  of  three  manifolds implies  that   a $C^0$-isometric  version  of the  Zimmer  program  is  also  true  for  singular geodesic  spaces closely  related  to three dimensional  manifolds,  namely  three  dimensional  geometric orbifolds  and  Alexandrov  spaces. A topological  version  of  the  Zimmer  Program is  seen  to  hold  in   dimension 3 for  Alexandrov  spaces  using  Pardon's ideas.  
\end{abstract}

%amsart \textit{Keywords: Zimmer  Program, 3-manifolds, Alexandrov Spaces, Hilbert  Smith Conjecture. 2020 AMS Classification: 57S05, 57S20, 22E40, 57M60, 53C23, 53C30.}

\section{Structure of the paper}

\tableofcontents

%%%%%%%%%%%%%%%%%%%%%%%%%%%%%
\section{Introduction}
%%%%%%%%%%%%%%%%%%%%%%%%%%%%%

\subsection*{Zimmers program}
The question on the nature  of  group  homomorphisms $\rho: \Gamma \to {\rm Diff} (M)$, between a  finitely generated group and the group of diffeomorphisms of a compact, $n$-dimensional, smooth manifold is interesting in many contexts. Particularly, in a series of conjectures known as the  Zimmer program \cite{zimmericm}, \cite{fisher1}, \cite{fisher2}, concerning the question on whether the group homomorphism cannot have large image if the dimension of the manifold is small, relative to the rank of the group. More precisely, the Zimmer program deals with groups $\Gamma$ which are lattices in a semisimple algebraic Lie group of rank at least 2, which we will refer to in this introduction as \textbf{higher rank lattices} see \ref{section:preliminaries}. As an example of this, in the recent result \cite{brownfisherhurtado}, it is found that a homomorphism $\rho : SL_{k+1}(\Z) \rightarrow {\rm Diff(M)}$ must have finite image when $k > n$ and $k \geq 2$, in this case the parameter $k$ is the rank of the Lie group $SL_{k+1}(\R)$. This result is greatly generalized for other higher rank semisimple lattices on \cite{brownfisherhurtado3}. Another instance of the Zimmer program is the complete characterization of the action in critical dimensions, c.f. Conjectures 1.2 and 1.3 in \cite {fisher1}. For example, when   the dimension hypothesis is  modified in the previous context to $n = k+1$, and the hypothesis that the action preserves a finite volume and an affine connection is added, then \cite[Theorem 6.9]{fisher1} tells us that the action is conjugated to the standard linear action of $SL_{n}(\Z)$ on $\mathbb{R}^n/\mathbb{Z}^n$. 

The  $C^{0}$-version of  the  Zimmer  Program, as  suggested  in \cite{weinberger}, and  \cite{ye1}, \cite{ye2}, asks  roughly  for changing the  category of manifolds and morphisms in the Zimmer Program, from the  smooth setting into  a  topological  setting, that is,  by  considering  a  group homomorphism from  a finitely  generated  group, and  specifically  a  higher rank  lattice, onto  the group  of  homeomorphisms  within a   prescribed  category (topological, smooth, piecewise  linear, etc.). The  following  Conjecture is  an  example  of  a  problem  stated  in  this  setting, found  in \cite{ye1}: 

\begin{problem}\label{problem:shengkui}
Any  group  action  of $SL_{k+1}(\mathbb{Z})$, with  $k\geq 2$, on  a  closed, aspherical $n$-manifold  by  homeomorphisms  factors  trough a  finite  group if $n < k$.  
\end{problem}

This variation of Zimmer Program is, as expected, much harder than the original one, and it has shown only small advances, such as the solution in the one-dimensional case by Hurtado and Deroin \cite{deroinhurtado}, and in the context of infinite homological actions on three-manifolds \cite{farbshalen}.

\subsection*{Variations of the problem}
It is natural to explore analogous rigidity results outside of the differentiable category and into the category of metric spaces endowed with extra structure. As an example of this, in \cite{Haettel} it is proved that any action of a higher rank lattice in a Gromov $\delta$-hyperbolic metric space must be elementary. Among many things, such result implies that any homomorphism $\rho :\Gamma \rightarrow \textrm{Mod}(S)$ from a higher rank lattice onto the mapping class group of a compact surface is finite, a result first proved in \cite{farbmasur}.

The notion closer to manifolds for  which  we explore these rifidity  results is that of Alexandrov spaces, which are metric spaces with a synthetic notion of curvature bounded from below. Alexandrov spaces include compact Riemannian manifolds and non-trivial modifications of them, such as orbifold quotients and Gromov-Hausdorff limits (see Section \ref{section:alex}. In this paper we propose the following variation of Problem \ref{problem:shengkui}

\begin{problem}\label{problem:shengkui_Alex}
Let  $X$  be an n-dimensional, compact, Alexandrov space. Does a group  homomorphism 
	\[ \rho: SL_{k+1}(\mathbb{Z}) \to {\rm Homeo}(X),	\]
factor through a finite group if $k > n$? Can we obtain a classification of such actions in the case $k = n$ or $k=n+1$ if we ask the action to be isometric?
\end{problem}

Of course, Problem \ref{problem:shengkui_Alex} can be stated for general higher rank lattices with comparison bounds on the dimension and rank as in \cite[Conjecture 4.12]{fisher1}.

Alexandrov spaces are rigid (in a sense analogous to Gromov's rigid geometric structures \cite{fisher1}) as their isometry groups are Lie groups with bounded dimension in terms of the dimension of the space \cite{galazguijarroisometry}. Moreover, 3-dimensional Alexandrov spaces which are sufficiently collapsed (c.f. section Alex), are in fact orbifolds over one of the eight Thurston geometries \cite{galazguijarrochu}. These reduction phenomenae together with classification results for isometric actions on Thurston geometries, lead us to the following much more tractable problem:

\begin{problem}\label{problem:3d-alex}
Let $X$ be a $3$-dimensional, compact Alexandrov space. Characterize any homomorphism
    \[  \rho: \Gamma \to {\rm Iso}(X), \]
where $\Gamma$ is a higher rank lattice.
\end{problem}

It is worth mentioning other rigidity results obtained for $3$-dimensional Alexandrov spaces such as the proof of the Borel  conjeture for sufficiently collapsed Alexandrov spaces \cite{barcenasnunezrmi}. Finally, John Pardon's proof of the Hilbert-Smith conjecture for three manifolds \cite{pardon}, can be extended to the singular case in the setting of Alexandrov spaces as it can be reduced to a local behaviour, leading to the following result:

\begin{theorem}\label{Theorem_Hilbert-Smith-3-man}
If $G$ is a locally compact, topological group, acting faithfully on a three dimensional Alexandrov space by homeomorphisms, then $G$ is a Lie group.
\end{theorem}

This result lead to the natural generalization of the Hilbert-Smith conjecture, which simply would ask if Theorem \ref{Theorem_Hilbert-Smith-3-man} is valid for $n$-dimensional Alexandrov spaces). A first approach to this generalizationis to extend the result of \cite{RepScep}, proving Hilbert-Smith conjecture for Lipschitz actions, where the difficulty lies on the extension  of Yang's Theorem (on the increase of dimension  in the  quotient for $p$-adic actions \cite{yang}  \cite{BREDON}) to Alexandrov spaces.

\subsection*{Main results and related discussions}

The main result of this paper, concerning Problem \ref{problem:3d-alex} is 

\begin{theorem}\label{main_thm}
Let $\widetilde{X}$ be a simply connected, homogeneous 3-dimensional manifold and let $H$ be a discrete group of isometries of $\widetilde{X}$, such that $X = \widetilde{X} / H$ has finite volume, then $X$ admits an infinite isometric action of a higher rank lattice $\Gamma \subset G$ if and only if the group $Iso(X)$ contains the group $SO(3)$. Moreover, the semisimple Lie group $G$ is isotypic of type $SO(3)$, the lattice is uniform and $X$ is a orbifold over either $S^3$ or $\mathbb{R} \times S^2$.
\end{theorem}

Recall that an isotypic group of type $SO(3)$ is an algebraic group which is, up to finite covers and connected components, a product of copies of $SO(3)$ and $SO(2,1)$ (see Section \ref{section:preliminaries}). In \cite{barcenasnunezrmi}, a three dimensional  Alexandrov space $X$ of ${\rm curv} \geq -1$ is  said  to  be \emph{sufficiently  collapsed},  if  there  exist $D>0$ and $\epsilon > 0$ such  that  the diameter of  $X$ is less  or  equal  to  $D$, and  the  volume  is  strictly  less  than  $\epsilon$. We  include  as  a  corollary  of the  results  here  a classification  of  the  discrete  groups acting  by  isometries  on  a three dimensional Alexandrov  space  with a  sufficiently  collapsed quotient

\begin{corollary}   
Assume  that  a discrete group $\Gamma$ acts  by  isometries on the  three dimensional  Alexandrov  space $X$ such  that the  quotient $X/\Gamma$ is  sufficiently  collapsed  with parameters $d$, and $\epsilon$. Then, Theorem \ref{zimmer:3man},  together  with the  geometrization of  $3$-dimensional  Alexandrov  spaces provide  a  classification of the  possible  such $\Gamma$ within  the lattices in the  isometry  groups.  
\end{corollary}

 As an immediate consequence of this theorem, we get the following corollary in the spirit of the Zimmer's problem

\begin{corollary}
Let $\Gamma$ be a higher rank lattice acting by isometries on a finite volume, three dimensional orbifold $X$ (modelled over a homogeneous 3-manifold $X$), then the action factors through a finite group if either:
\begin{itemize}
    \item $X$ is aspherical or,

    \item $\Gamma$ is non-uniform.
\end{itemize}
As an example of this, we have $\Gamma = SL_{r}(\mathbb{Z})$ with  $r\geq 3$.
\end{corollary}

The proof of Theorem \ref{main_thm}  relies on close, case by case examination of Thurston's $3$-dimensional geometries, their finite volume quotients and their corresponding isometry groups. The computations of such groups can be summarized in the following
 
\begin{theorem}\label{zimmer:3man}
Let $\widetilde{X}$ be a simply connected, homogeneous 3-dimensional manifold and let $G$ be a discrete group of isometries of $\widetilde{X}$, such that $\widetilde{X} / G$ has finite volume. Then the isometry group $Iso(\widetilde{X}/G)$ has finitely many connected components, such that its connected component of the identity is isomorphic to
\begin{itemize}
		\item a closed subgroup of $\IS^1$, if $\widetilde{X}$ is either $\IH^2 \times \IR$, $\widetilde{SL}_2(\IR)$ or $Nil$;

		\item a closed subgroup of $SO(3) \times S^1$, if $\widetilde{X} =  S^2 \times \IR$;

		\item a closed subgroup of $\IR^3 / \IZ^3$, if $\widetilde{X} = \IR^3$;

		\item a closed subgroup of $SO(4)$, if $\widetilde{X} = S^3$. 
\end{itemize}
Moreover, $Iso(\widetilde{X}/G)$ is finite if $\widetilde{X}$ is either $\IH^3$ or $Sol$.
\end{theorem}

\subsection*{Strategy of the proof and structure of the paper.}
Main Theorem \ref{main_thm} is proved in Section \ref{section:preliminaries} using the classification of isometry groups of orbifolds given by Theorem \ref{zimmer:3man}, together with rigidity results of semisimple Lie groups. To prove Theorem \ref{zimmer:3man}, one first need to understand finite volume quotients of Thurston's three-dimensional geometries. Among such geometries, the most homogeneous ones are $S^3$, $\IH^3$ and $\IR^3$; and the remaining five present a more flexible description as fiber bundles
	\[	F \rightarrow X \rightarrow B,	\]
where $B$ is a two dimensional homogeneous geometry for $X$ either $\IH^2 \times \IR$, $\widetilde{SL}_2(\IR)$ or $\textrm{Nil}$ and $B = \IR$ for $X$ either $Sol$ or $S^2 \times \IR$. In this context, a discrete group acting on the homogeneous space $X$, acts on the base space of the corresponding fiber bundle as well. The induced action on the base space of the fiber bundle presents a dual behaviour given by the following Theorem, whose proof can be derived from the discussions on \cite{thu}.

\begin{theorem}\label{theo:discrete}
Let $G$  be  a discrete group of isometries of any of the 3-dimensional geometric manifolds $\IH^2 \times \IR$, $\widetilde{SL}_2(\IR)$ or $Nil$; then either 
	\begin{itemize}
		\item $G$ projects to a discrete group of isometries of the base $B$ of the fiber bundle, or 

		\item the orbifold $\widetilde{X}/G$ has infinite  volume. 
	\end{itemize}
Moreover, in the cases $Sol$ and $S^2 \times \IR$, the projection to the base space is always discrete.
\end{theorem}

For the sake of completeness we present here a proof of Theorem \ref{theo:discrete}. The proofs of Theorem \ref{zimmer:3man} and Theorem \ref{theo:discrete} are carried out in a case by case setting on each Thurston geometry.

The structure of the paper is as follows: In Section \ref{section:preliminaries} we present background material on three dimensional Alexandrov spaces, the Hilbert Smith conjecture and semisimple Lie groups and their lattices. Sections \ref{section:euclidean} through \ref{section:SL} cover the proof of Theorem \ref{theo:discrete} and Theorem \ref{zimmer:3man} on each individual three dimensional geometry.

Theorem \ref{theo:discrete} can be used to obtain explicit characterizations of discrete groups acting on the corresponding three-dimensional geometry (in particular within the $Nil$ and $Sol$ cases), which lead to the proof of Theorem \ref{zimmer:3man} in each case.

\subsection{Concluding remarks  and open questions}

In this work we  proposed  a  metric variation  on  the  Zimmer program. The  variation  consisted in 
\begin{itemize}
\item Strengthening the  category of automorphisms of  the  action from $C^{0}$ to  isometries. 
\item Relaxing  the topological type  of  the spaces  considered from  smooth manifolds to   Alexandrov  spaces,  which  include three dimensional geometric orbifolds. 
\end{itemize}

While  Alexandrov  spaces  have  an  open  dense  subset  which  is a topological  manifold, results   related  to Zimmer's  conjeture  do  not  apply  directly  because  the  topological   manifold  is  open, and  the  only rigidity  results  for  actions  on  open  manifolds  in the  spirit of  the  Zimmer program which  are  known  to  the  authors  are  restricted to  the one  dimensional  case \cite{deroinhurtado}.  In  another  instance  of  a  complication, the  manifold  is  not Riemannian, as  Otsu-Shioya's  example  shows \cite{OtsuShioya}.

There  exist several instances  of families  of homeomorphisms of metric  geodesic  spaces for  which  rigidity results  of  actions  of discrete  groups can  be  proved. Among  them  we  can  consider,  for  a  strengthening with  respect  to homeomorphisms  and  a weakening  with  respect  to  isometries: 

\begin{itemize}
\item Quasiconformal  homeomorphisms, as  in the  alternative  proof  of  Mostow Rigidity Theorem  by  \cite{bourdon}. 

\item Bilipschitz homeomorphisms, as  in  the  proof  of the  Hilbert-Smith Conjecture  mentioned  before
\cite{RepScep}. 

\item Quasim\"obius homeomorphisms  as  in \cite{bonkkleiner},  where  rigidity  results  for  them  have  as  a consequence  the  rigidity  of actions  of quasi  convex  cocompact  actions  on ${\rm CAT(-1)}$-spaces.
\end{itemize}

Moreover,  the specific analytic and  geometric  characteristics  of  the  class  of  homeomorphisms  are  exploited  in  the  process of proving  an  action rigidity  result in an analogous  way to how we  used  the geometric  structure  of Alexandrov three  spaces  in this  work,  inspired  by  the  proofs  of  Zimmer  program results  in  the  diffeomorphism case.  

Let us  introduce  the  notation  
$${\rm AUT}^{\alpha}(X) $$
 for homeomorphisms  of a  metric  space  with a  metric  property $\alpha$ and  let us  refer  to   a  metric  condition $\alpha$ as  a  decoration in  analogy  with  surgery  theory,  having  at  least the   following  examples  in mind: 
 \begin{enumerate}
     \item The  smooth  case $\alpha= {\rm Diff}$, for  diffeomorphisms of  the  smooth  structure associated  to a Riemannian  metric  on a smooth  manifold with fixed  metric. 
     \item The  topological case $\alpha={\rm Top}$, refering  to  homeomorphisms of a  topological  manifold. 
     \item The  isometric  case $\rm Iso$, meaning   isometries  of the  geodesic metric  space associated  to a geometric three manifold, orbifold  or  Alexandrov  space.   
     
    \item The  quasiconformal  case $\alpha= QC$,  the  referring  to quasiconformal   homemomorphisms, $\alpha=QM$, associated  to  quasim\"obius  homeomorphisms  and $\alpha=BI$ for  bilipschitz  homemorphisms of a  geodesic  length metric  space as  discussed  in  the  paragraph above.

 \end{enumerate}

We  can  consider  the problem  of  describing  the behaviour  of  a group homomorphism  

$$ \Lambda \longrightarrow {\rm AUT}^{\alpha}(X) $$
to  a homeomorphism group  with  decorations as  described. 
Notice  that  $\alpha= {\rm Diff}$ is  the  Zimmer  program as  described for  instance  in \cite{fisher1},\cite{fisher2}, \cite{zimmerbook}. On the  other  hand  $\alpha= {\rm Top}$ is  the  $C^{0}$-Zimmer  program as  described in \cite{weinberger}, \cite{ye1}, \cite{ye2}. Finally $\alpha = {\rm Iso}$ was  the  point  of  view  adopted  in this  note, and  $\alpha=QC, QM, BI$ are  as  before. 

We  would  like  to  finish the  present  note  with the  following  two  questions: 
\begin{itemize}

\item  To what extent a condition of prescribed curvature in metric spaces, such  as  the  Alexandrov  condition, or  a  choice  of $QC$, $QM$ or  $BI$-structures can be seen  as a  rigid  structure,  in  the  sense  of Gromov \cite{fisher1}? 

\item For  which  decorations  of  homeomorphisms  $\alpha$ is  it possible to prove  that homeomorphisms of a higher  rank lattice $\Lambda $  with  respect  to  the  dimension of  an Alexandrov  space $X$ 
$$ \Lambda\longrightarrow {\rm AUT}
^{\alpha }(X)$$
or  in general a metric measure space $X$ with  finite  Hausdorff dimension less  than  the  rank of  $\Lambda$ either  factorize trough  a  finite  quotient of  $\Lambda$  or  produce  a  quotient  of infinite  volume? 
\end{itemize}

\section*{Acknowledgments}

The  first  author  thanks DGAPA Project IN100221.  Parts of  this  project  were written  during  a  Sabbatical  Stay  at  the Universit\"at  des Saarlandes,  with  support  of DGAPA-UNAM Sabbatical Program and  the  SFB TRR 195 Symbolic Tools in Mathematics and their Application,  and CONACYT trough grant CF 217392. 

The  second  named  author   thanks 
CONACYT Grant  CB2016-283988-F,  
CONACYT Grant  CB2016-283960,  
as  well as  a  DGAPA-UNAM  Postdoctoral Scholarship.  

\section{Preliminaries}\label{section:preliminaries}

\subsection{Three Dimensional Manifolds}
Recall  that  a closed  three  dimensional manifold $M$  is  prime  if  it cannot  be written  as  a  connected  sum with summands not  homeomorphic to  the  three dimensional  sphere.

By the  prime  decomposition  Theorem, 3.15  page 31 in \cite{hempel},  any closed three dimensional  manifold  can be  written as  a  connected  sum of  prime  factors. 

Recall that  a  model geometry is a simply connected smooth three  manifold $X$ together with a transitive action of a Lie group $G$ on $X$ with compact stabilizers such that  $G$ is maximal among groups acting smoothly and transitively on $X$ with compact stabilizers. 

A geometric structure on a  three dimensional manifold $M$ is a diffeomorphism from $M$ to $X\diagup \Gamma $ for some model geometry $X$, where $\Gamma$ is a discrete subgroup of $G$ acting freely on X ;

It  is  a  consequence  of  Thurston geometrization  theorem due  to Perelman,  that any prime closed  three dimensional  manifold  can be  cut  along  three dimensional  tori such that  the interiors  of  the resulting  manifold  carries  a  geometric  structure of finite volume. 

The  eight  three dimensional geometries  which  admit  at  least  one  compact  three manifold  are : 
\begin{enumerate}
    \item Spherical, where the  three dimensional  sphere  is a  representative. 
    \item Euclidean, where the  flat  three dimensional  torus  is  an  example,
    \item Hyperbolic, where  the three dimensional hyperbolic space  is  an example. 
    \item $\widetilde{SL}_{2}(\mathbb{R})$, where  an  example  is given by  the  unit  tangent bundle in a   Riemannian metric  of  the tangent  bundle  over a genus  two  surface.
    \item $Nil$,  where  an example  is  given  by  the mapping  torus  of  a  Dehn twist  on the two  dimensional  torus. 
    \item $Sol$,  where  an  example  is  given  by a  manifold  which fibers over the line with fiber the plane. 
     \item $S^{2}\times \mathbb{R}$.
    \item $S^{1}\times \mathbb{H}^{2}$. 
    
\end{enumerate}

A closed 3-manifold has a geometric structure of at most one of the 8 types above.  

\begin{remark}(Non-uniqueness of Geometric Structures on three  manifolds)
\begin{itemize}

    \item Finite volume non-compact manifolds  may have more than one type of geometric structure. An  example is  the  complement  of  the  trefoil knot, which  has hyperbolic  structure and  $\widetilde{Sl}_{2}$ structure.
    \item If  the three manifold   has  infinite  volume, it  might  carry  many  geometric  structures, for  example $\mathbb{R}^{3}$  is  diffeomorphic  to  all  aspherical  model  geometries. 
    \item There  exists  an  infinite  number  of geometric structures with no compact models; for example, the geometry of many non-unimodular 3-dimensional Lie groups, see remark \ref{remark:unimodular} below. 
\end{itemize}

\end{remark}

\subsection{Three dimensional Alexandrov Spaces}\label{section:alex}
We  will need  some  preliminaries  on  three  dimensional Alexandrov  spaces. For  a  general  reference  see \cite{buragobook}. 

For  the  purposes  of  this  work, Alexandrov spaces are well behaved metric spaces which have three main properties that we want to highlight here: 
\begin{enumerate}    
    \item They  have  an open dense subset which is topological manifold.

    \item Their isometry groups are Lie groups.

    \item The class  of  Alexandrov  spaces  is stable under Gromov-Hausdorff convergence. 

    \item They   include orbifolds over Riemannian manifolds. 
\end{enumerate}

In  slightly  more  detail, Alexandrov  spaces are  a  synthetic  generalization  of complete Riemannian  manifolds  with  a  lower  bound  on sectional  curvature. The  generalization  uses  comparison  triangles with  respect  to  the model spaces  $S_{k}^{2}$,  which  are simply  connected,  two  dimensional  complete Riemannian  manifolds of  constant  curvature  $k$. More precisely, for  $k>0$, $S_{k}^{2}$ is  the sphere  of  radius  $\frac{1}{\sqrt{k}}$, for  $k<0$, $S_{k}^{2}$ is  the  hyperbolic  plane $\mathbb{H}^{2}(\frac{1}{\sqrt{-k}})$  of  constant  curvature $k$, and  for $k=0$, $S_{k}^{2}$ is  the  euclidean space  $\mathbb{R}^{2}$.

Given  a  geodesic  triangle in a geodesic length  space $(X,d)$, with  vertices  $p, q, r \in X$, a  comparison  triangle in $S_{k}^{2}$  is a geodesic triangle $\bar{p}\bar{q}\bar{r}$ having the same side lengths. The geodesic length  space $(X,d)$ is  said  to  satisfy  the Topogonov property for  $k \in \mathbb{R}$, if for  each  triple $p,q,r \in X$ of vertices of a  geodesic  triangle, and each point $s$ on the  geodesic  from $q$  to $r$, the inequality  $d(p,s)\geq d(\bar{p}, \bar{s})$ holds,  where $\bar{s}$  is  the point  on the geodesic side  $\bar{q}\bar{r}$ of  the  comparison  triangle  with $d(\bar{p}\bar{s})=d(p,s)$. 

\begin{definition}\label{def:alex}
A $n$-dimensional $k$-Alexandrov  space is a  complete, locally compact, length space of finite Hausdorff dimension $n$, such that the Topogonov Property is satisfied locally for $k$. 
\end{definition}

Topogonov's  globalization theorem tells us that the local and global Toponogov property are equivalent in $k$-Alexandrov spaces. By  Gromov's  precompactness  theorem, Alexandrov $n$-dimensional  spaces  arise  as  Gromov-Hausdorf  limits  of  compact  riemannnian  manifolds  of  dimension  $n$  for  which  the  sectional  curvature is  bounded  below  by  $k$,  and  the  diameter is bounded  above  by  some   fixed  positive  number  $D$. 

The  class  of  $k$-Alexandrov   spaces includes   riemannian   manifolds  of  sectional  curvature  bounded  below  by  $k$, and  several  constructions  including more general  geodesic  length spaces    such  as  euclidean  cones, suspensions, joins, quotients  by  isometric  actions of  compact  Lie  groups, and  glueings  along  a  submetry, see \cite{galazglance} section  2.2. From  now  on,  we will  omit  the  $k$  from  the  notation. 

There  exists  a  notion  of  angle  between  geodesics of   an  Alexandrov space,  and  a  space  of  tangent  directions at  a  given  point  $p$, denoted  by $\Sigma_{p}$, can  be  defined  as  the  completion  of  the metric  space  of  equivalence  classes  of  geodesics  making  a  zero  angle.

 The   space  of  tangent  directions  at  a  point $p$ in an Alexandrov space $X$, denoted  by $\Sigma_{p}$, has the structure of a $1$-Alexandrov space of  Hausdorf  dimension ${\rm dim} (X)-1$. There is a set $R_X \subset X$, called the \emph{set  of  metrically regular  points}, where a point $p$ belongs to $R_X$ if its direction space $\Sigma_{p}$ is isometric to  the radius  one  sphere. The  complement is  called the {\emph set  of  metrically  singular  points} and denoted by $S_X = X \setminus R_X$. There are examples of Alexandrov spaces whose space of metrically singular points is dense, as seen in an example constructed in \cite{OtsuShioya} as a limit of Alexandrov spaces, using baricentric subdivisions of a tetrahedron. However, for every Alexandrov space $X$, there is a dense subset of topologically regular points, whose space of directions are \emph{homeomorphic} to a sphere (the  set  of topologically singular points  is  the  complement  of  the  set  of  topologically  regular  points). By  Perelman's conical  neighborhood theorem,  every  point  $p$  in an  Alexandrov  space  has  a neighborhood  pointed  homeomorphic  to  the  euclidean  cone  over  $\Sigma_{p}$, so that a locally compact, finite dimensional Alexandrov space has a dense subset which is a topological manifold.

In the  specific  case  of   dimension  three, there are only two possibilities for the homeomorphic type of the space of directions, which is the two sphere $S^2$, for the topologically regular points and the real projective space $\mathbb{RP}^2$ for the topologically singular points. Let  us summarize the  basic  structure  of  three  dimensional  Alexandrov  spaces due  to  Galaz-Garc\'ia  and  Guijarro, compare  Theorem  1.1  in page  5561  of  \cite{galazguijarrosurvey}, and Theorem  3.1  and  3.2    in page 1196  of  \cite{galazguijarroisometry}. See also \cite{harvey-searle}.

\begin{theorem}\label{theo:alexdim3}
Let $X$ be a three dimensional Alexandrov space.

\begin{itemize}
\item  The set  of metrically  regular  points  is  a Riemannian three manifold. 
\item The set of topologically singular points is a discrete subset of $X$.
\item If $X$ is closed, and positively curved Alexandrov space, that contains a topologically  singular  point, then $X$ is  homeomorphic  to  the suspension  of  $\mathbb{R}P^{2}$. 
\end{itemize}
\end{theorem}

A  closed  Alexandrov  space  is \emph{geometric} if  it can be  written  as  a  quotient  of  one  of  the  eight  geometries  of  Thurston under  a cocompact  lattice. The  following  theorem  was  proved  as  Theorem 1.6  in  \cite{galazguijarrosurvey}  in  page  5563. See also \cite{harvey-searle}

\begin{theorem}\label{theo:geometrizationalex}
A  three  dimensional  Alexandrov  space admits  a  geometric decomposition into geometric pieces, along spheres, projective  planes, tori  and  Klein  bottles.

\end{theorem}

We  now  direct  our  attention  to the   isometry  group  of  three dimensional  Alexandrov  spaces. 

\begin{theorem}\label{theo:isometryalex}
Let  $X$  be  an  $n$-dimensional  Alexandrov  space  of  Hausdorff dimension  $n$. 
\begin{itemize}
\item The  Isometry group  of $X$  is  a   Lie  Group which is compact if $X$ is compact as well. 
\item The  dimension  of   the  group  of  Isometries of  $X$  is at  most   
$$ \frac{n(n+1)}{2},$$
and  the bound  is  attained  if  and  only if  $X$  is  a  Riemannian  manifold.  
\end{itemize}

\end{theorem}
\begin{proof}
\begin{itemize}
\item The  first  part is  proved as  the  main Theorem, 1.1  in \cite{fukayayamaguchi}. The  second  part  follows  from  the  Van Dantzig-Van der Waerden  Theorem \cite{dantzig}.   
\item This  is  proved  as Theorem 3.1 in page  570.   
\end{itemize}
\end{proof}

\begin{remark}

It  is  proved  in  \cite{bagaevzhukova} that  the  same   lower bound for  the  dimension  of  the  isometry  group  holds  in general   for  Riemannian orbifolds. 

\end{remark}

\subsection{Hilbert-Smith  Conjecture}\label{section:hilbert-smith}

The  following  conjecture  was formulated  as  an  extension of Hilbert's  5th Problem: 

\begin{conjecture}[Hilbert-Smith conjecture]
If $G$ is a locally compact, topological group, acting faithfully on a topological manifold, then $G$ is a Lie group.
\end{conjecture}

See  \cite{taofifth} for  a  modern  account. 

As a consequence of structural theorems of locally compact groups, such as  the Gleason-Yamabe theorem and its  predecessor by  Von  Neumann \cite{vonneumann}.   a counter-example to  the  Hilbert-Smith conjecture   must contain a copy of a p-adic group $\widehat{\Z}_p$, for some $p$, see \cite{LeeJoo}, thus giving the equivalent conjecture

\begin{conjecture}[Hilbert-Smith conjecture $p$-adic version]\label{conjecture:p-adic_HS}
For every prime $p$, there are no faithful actions of the $p$-adic group $\widehat{\Z}_p$ on a topological manifold.
\end{conjecture}

Conjecture \ref{conjecture:p-adic_HS} has been proven in different contexts. For example, if there is a notion of dimension which must be preserved, such as bi-Lipschitz actions of $\widehat{\Z}_p$ on Riemannian manifolds, where three notions of dimension coincide: Hausdorff dimension, cohomological dimension with integer coefficients and topological dimension. In such setting, the bi-Lipschitz condition tells us that the Hausdorff dimension on the quotient cannot decrease, but on the other hand a theorem by Yang \cite{yang}, tells us that the cohomological dimension of the quotient increases by two, leading to the following result:

\begin{theorem}[Repov\v{s}-\v{S}\v{c}epin \cite{RepScep}]\label{Theorem:RepScep}
There are no faithful actions by bi-Lipschitz maps of the $p$-adic group $\widehat{\Z}_p$ on a Riemannian manifold. 
\end{theorem}

The stronger setting of topological actions is much harder and has been proven only for small dimensions

\begin{theorem}[\cite{pardon1}, \cite{pardon}]\label{Theorem:pardon}
For every prime $p$, there are no faithful actions by homeomorphisms of the $p$-adic group $\widehat{\Z}_p$ on a topological manifold of dimension $n \leq 3$.
\end{theorem}

\begin{remark}
The $p$-adic group can be described as
	\[	\widehat{\Z}_p = \left\{	\sum_{n = 0}^\infty a_n p^n : a_n \in \{0, 1, \cdots, p-1\},	\right\}  	\]
so that $p^k \widehat{\Z}_p \subset \widehat{\Z}_p$ is an open, normal subgroup, with $\widehat{\Z}_p / p^k \widehat{\Z}_p \cong \Z_{p^k \Z}$, giving the inverse limit description $\widehat{\Z}_p = \lim_{\leftarrow} \Z_{p^k}$, moreover, the group $\widehat{\Z}_p$ is homeomorphic to the Cantor space $\{0, \cdots,p-1\}^\mathbb{N}$. Observe that there is a topological $2$-manifold with the cantor space $2^\mathbb{N}$ as its ends space, which is $\Sigma = S^2 \setminus C$, where  $C \subset S^2$ is a closed subset homeomorphic to $2^\mathbb{N}$. Thus, there is a faithful action of $\widehat{\Z}_2$ on $End(\Sigma) \cong 2^\mathbb{N}$ and every homeomorphism of $End(\Sigma)$ extends to a homeomorphism of the surface $\Sigma$, however, by Theorem \ref{Theorem:pardon}, such extensions cannot be promoted  to an  action of  $\widehat{\Z}_2$ on the  Freudenthal compactification.
\end{remark}

Hence, the weaker version of the $p$-adic Hilbert-Smith conjecture for Alexandrov spaces holds, and  we can  consider  the  following conjecture: 

\begin{conjecture}
If $G$ is a locally compact, topological group, acting faithfully on a finite dimensional Alexandrov space by homeomorphisms, then $G$ is a Lie group.
\end{conjecture}

A  consequence of Theorem \ref{Theorem:pardon}, gives us a  three dimensional  case  of  this result \ref{Theorem:HS-for-Alex}

\subsection{Lie Groups and  Lattices}
A  real  Lie  group  is a Hausdorff topological  group  which is a  $C^{\infty}$ smooth  manifold for  which  the  multiplication  and  inversion  are  smooth  maps.

Recall  that  by  Haar's  Theorem  there  exists  up to a positive multiplicative constant, a unique countably additive, nontrivial measure $\mu$ on the Borel subsets of $G$ which is  right  translation  invariant, has  finite  values  on  compact  subsets, and  is inner  and  outer regular. 

\begin{definition}
    A  lattice in a  Lie  Group  $G$  is  a  discrete  group $\Lambda$ for  which the quotient space  $G/\Lambda$ has  finite measure. 
    \end{definition}

The  right  translate $\mu(g^{-1}\quad )$by  an  element  in  the group $g^{-1}$   of a  right  invariant  Haar  measure $\mu$ is  a right  invariant  Haar  measure,  and  hence  there  exists   a  real    function $\Delta$ satisfying
$$ \mu(g^{-1}S) = \Delta(g)\mu (S).$$

\begin{definition}
    A  group  is  said   to be  unimodular  if  the  function $\Delta$  is the constant   function  $1$. 

\end{definition}

\begin{example}
The  following  families  of  examples   of  Lie  groups  will   be  the  main focus   of  the  article. 
\begin{itemize}
\item By the  second  Myers-Steenrod  theorem \cite{myers-steenrod},  the  isometry  group  of  a  smooth  manifold  is a  Lie  group. 

\item By  the  Montgomery-Zippin theorem \cite{montgomeryzippin},   if a topological group  acts by  isometries  transitively  on  a  finite  dimensional, locally  compact, connected  and  locally  connected  metric  space, then  it  is  a  Lie  group.
\item By  results  of  Bochner \cite{bochner}, the  isometry  groups  of a smooth manifold  of  constant  negative  Ricci  curvature  is  finite.

On the other  hand, a locally compact subgroup of $C^{2}$ diffeomorphisms  of  a  $C^{2}$ manifold  for  which  the  trivial  subgroup is  the  only  subgroup  with  fixed points  with  nonempty  interior must  be a  Lie  group  \cite{bochnerlocal}.
\end{itemize}
\end{example}

In  the  subsequent  sections  of  this  article,  we  will  examine  the  isometry  groups  of  three  manifolds  and  their  lattices, as  well  as  the  isometry  groups  of  orbifolds  or Alexandrov  spaces.

\subsection{Discrete groups of isometries} 

If $\widetilde{X}$ is a complete, simply connected, Riemannian manifold and $\Gamma \subset Iso(\widetilde{X})$ a discrete subgroup of isometries, then $X / \Gamma$ has the structure of a complete Riemannian orbifold. The covering map $\rho : \widetilde{X} \rightarrow X$ satisfies the property that $\rho(x) = \rho(y)$ if and only if $\gamma x = y$ for some $\gamma \in \Gamma$. An isometry $\phi : X \rightarrow X$ lifts to $\widetilde{\phi} : \widetilde{X} \rightarrow \widetilde{X}$ such that $\rho \circ \widetilde{\phi} = \phi \circ \rho$ and for every $\gamma \in \Gamma$ and $x \in \widetilde{X}$ we have
	\[	\rho \circ \widetilde{\phi}(\gamma x) = \phi \circ \rho(\gamma x) = \phi \circ \rho(x) = \rho \circ \widetilde{\phi}(x),	\]
thus there exist $\gamma' \in \Gamma$ such that
	\[	\widetilde{\phi}(\gamma x) = \gamma' \widetilde{\phi} (x),	\]
that is $\widetilde{\phi} \circ \gamma \circ \widetilde{\phi}^{-1} = \gamma'$ and $\widetilde{\phi} \in N = N_{Iso(\widetilde{X})}(\Gamma)$. This tells us that we have the isomorphism 
	\[	Iso(X) \cong N_{Iso(\widetilde{X})}(\Gamma) / \Gamma.	\]

\begin{proposition}\label{prop:centralizer-normalizer}
If $G$ is a Lie group and $\Gamma \subset G$ is a discrete subgroup with associated normalizer and centralizer subgroups
	\[	N = N_G(\Gamma), \qquad Z = Z_G(\Gamma),	\]
then the connected components of $N$ and $Z$ coincide. Moreover, if $Z_0$ denotes such connected component, the projection $\pi : N \rightarrow N/\Gamma$ is a covering Lie group homomorphism such that $\pi(Z_0) \subset N/\Gamma$ is the connected component of the identity.
\end{proposition}

\begin{proof}
If $g_t \in N$ is a 1-parameter subgroup and $\gamma \in \Gamma$, then $g_t \gamma g_{-t} = \gamma_t$ is a 1-parameter group in $\Gamma$, but as $\Gamma$ is discrete, $\gamma_t = \gamma$ and this tells us that $g_t \in Z$, so that $Z_0 = N_0$. Now, $N$ is a Lie group having $\Gamma$ as a normal, discrete subgroup so that the projection map
	\[	\pi : N \rightarrow N/\Gamma	\]
is a homomorphism of Lie groups and a covering map. In particular, $\pi(N_0)$ is a connected, open Lie subgroup of the same dimension of $N/\Gamma$ and thus it is the connected component of the identity.
\end{proof}

\subsection{Lattices on semisimple Lie groups of higher rank}\label{section:lattices}

Recall that an algebraic $\R$-group is a subgroup $\mathbb{G}_\C \subset GL_m(\C)$ obtained as solutions of polynomial equations with coefficients over $\R$ and $\mathbb{G}_\R = \mathbb{G}_\C \cap GL_m(\R)$ is a real Lie group. In this context we say that $\mathbb{G}_\R$ is a real form of $\mathbb{G}_\C$ or that $\mathbb{G}_\C$ is a complexification of $\mathbb{G}_\R$. The local structure of a Lie group is captured by its Lie algebra, so that two groups are locally isomorphic if and only if they have isomorphic Lie algebras, and thus, they can be obtained one from the other by taking connected components and topological covers.

The class of semisimple Lie groups can be  defined  as  the  class of  Lie groups which are constructed up to covers and connected components from algebraic $\R$-groups which split as products $G_1 \times \cdots \times G_k$, where each factor $G_j$ is simple.  This  definition is  equivalent  to  other  definitions  of  semisimple Lie groups aviailable  in  the  literature, see  \cite{zimmerbook}. 

\begin{remark}
Not every semisimple Lie group is an algebraic group as the group $SL_2(\R)$ has a universal cover, denoted by $\widetilde{SL}_2(\R)$, which is  homeomorphic to $\R^3$ and it cannot be embedded in any linear group $GL_m(\C)$ as a Lie subgroup. In the same way, not every semisimple Lie group splits as a product of simple Lie groups, as the example $SO(4)$ shows, but its universal cover is isomorphic to the product $SU(2) \times SU(2)$. In general, given a connected semisimple Lie group $G$, with center $Z(G)$, then the quotient $G/Z(G)$ is a connected, linear algebraic group which splits as a product of simple groups and it is locally isomorphic to $G$. Thus it is common for some results to ask for the group to be centerless.
\end{remark}

In the context of algebraic groups defined over a field $k$, the concept of $k$-rank is the maximal abelian subgroup which can be diagonalized over $k$. Thus, for a complex algebraic group, the $\C$-rank is the dimension of a maximal subgroup isomorphic to a complex torus $(\C^*)^l$ and we are particularly interested in the real rank of a real form. We can observe that the real rank of a product $G_1 \times \cdots \times G_k$ is the sum of the real rank of its factors $G_j$ and we can give some explicit examples.

\begin{example}
The following is a complete list, up to local isomorphism, of complex, simple Lie groups and some examples of their real forms:
	\begin{enumerate}
		\item The group $SL_n(\C)$, has $\C$-rank $n-1$ and has the groups $SU(p,q)$ and $SL_n(\R)$ as real forms, with real rank equal to $\mathrm{min}\{p,q\}$ and $n-1$ respectively.

		\item The group $SO(n,\C)$ has $\C$-rank $\left\lfloor \frac{n}{2} \right\rfloor$ and has the groups $SO(p,q)$ as real forms, having real rank equal to $\mathrm{min}\{p,q\}$.

		\item The group $Sp(2n,\C)$ has $\C$-rank $n$ and has the groups $Sp(p,q)$ and $Sp(2n,\R)$ as real forms, with real rank equal to $n$ and $\mathrm{min}\{p,q\}$ respectively.

		\item The exceptional complex groups $G_2(\C)$, $F_4(\C)$, $E_6(\C)$, $E_7(\C)$, $E_8(\C)$ have $\C$-rank determined by the corresponding subindex.
	\end{enumerate}
\end{example}

\begin{remark}
Between the possible real forms of a complex semisimple Lie group, there is one and only one compact real form up to conjugacy and such compact form has a compact universal cover, so the compactness property survives in the process of passing to a cover. We can thus, speak of the compact factors of a real semisimple Lie group. Moreover, the rank of a compact Lie group, defined as the dimension of a maximal torus $(S^1)^l$ contained in the group, equals the rank of its complexification and has real rank equal to $0$. Finally, given a compact, connected, Lie group $C$, there is a finite cover of $C$ that splits as $G \times T$, where $G$ is an algebraic semisimple Lie group, and $T$ is a torus. 
\end{remark}

\begin{definition}[Higher Rank Lattice]
A semisimple Lie group is said to have higher rank if its real rank is greater than or equal to $2$. Moreover, if a semisimple Lie group has a complexification whose simple factors are all locally isomorphic, the group is called isotypic.
\end{definition}

Isotypic Lie groups are important  because we can construct  irreducible lattices in them,  which don't split as a product of lattices in the simple factors.

\begin{example}
If $\sigma : \mathbb{Q}(\sqrt{2}) \rightarrow \mathbb{Q}(\sqrt{2})$ is the non-trivial Galois automorphism and $Q(x,y,z,t) = x^2 + y^2 - \sqrt{2} z^2 - \sqrt{2} t^2$, $\sigma(Q) =  x^2 + y^2 + \sqrt{2} z^2 + \sqrt{2} t^2$. The groups $G = SO(Q,\R) \cong SO(2,2)$ and $K = SO(\sigma(Q),\R) \cong SO(4)$ are semisimple Lie groups, with $K$ compact and $G$ of real rank equal to $2$. If we consider the integral points in $G$, that is, the group $\Gamma = SO(Q,\Z(\sqrt{2})) \subset G$, then the group,
	\[	\widehat{\Gamma} = \{ (g , \sigma(g)) \in G \times K : g \in \Gamma \}	\]
is discrete. In fact there is an $\R$-group $\mathbb{H}_\C \subset GL_m(\C)$ such that $\mathbb{H}_\R = G \times K$ and $\mathbb{H}_\Z = \widehat{\Gamma}$, in particular, it is a lattice which is co-compact. As the projection $G \times K \rightarrow G$ has compact Kernel and maps $\widehat{\Gamma}$ onto $\Gamma$, thus $\Gamma \subset G$ is discrete and thus, a co-compact lattice in $G$.
\end{example}

\begin{remark}
The previous example captures the general behaviour of irreducible lattices in isotypic semisimple Lie groups.  In fact, isotypic, semisimple Lie groups are the only cases of semisimple Lie groups admiting irreducible lattices and such lattices are constructed with the method of the previous example, but with higher degree extension fields $k/\mathbb{Q}$. See \cite{morris}, Section 5.6 for the details of the previous example and the construction in general.
\end{remark}

\section{Detailed Strategy  for  the Proof  of Main  Theorem }\label{section:strategy}
We are interested in the particular case where  given a  three-manifold  $X$, the  universal cover $\widetilde{X}$ is a homogeneous space, i.e. its group of isometries acts transitively  on $\widetilde{X}$ , and $X$ has finite volume. A general setting where this is achieved is when we consider $G = \widetilde{X}$ a simply-connected Lie group with a right-invariant (or left-invariant) riemannian metric and $X = G/\Gamma$, with $\Gamma \subset G$ a lattice subgroup (i.e. $\Gamma$ is a discrete subgroup such that $G/\Gamma$ has finite, left $G$-invariant volume). As there is an embedding $G \subset Iso(G)$, we have that $N_G(\Gamma) \subset N_{Iso(G)}(\Gamma)$, but it could happen that $N_G(\Gamma) / \Gamma$ is strictly smaller than $Iso(G/\Gamma)$. On the other hand, we can extend $\Gamma$ to a discrete subgroup of $Iso(G)$ which is not completely contained in $G$, so that the isometry group $Iso(G/\Gamma)$ is decreased.

In the following sections, we will examinate this phenomenon in the Thurston Geometries, and determine the possible isometry groups of the corresponding finite-volume orbifolds.

The  main result \ref{zimmer:3man} is proved by  a  case by  case schema organized around  the  eight geometries, which  we present  in a  resumed form  below. 

Within the    most  homogeneous three of  them (spherical, hyperbolic and  euclidean), the  Theorem is   a consequence  of  classical  results,  which  we  gather  from  references and  include the  classification of  spherical  manifolds due to Seifert-Threlfall, Borel's density  Theorem for hyperbolic  geometry, and  Bieberbach's  theorems for the  euclidean case, respectively. 

In the  cases  of  $Nil$ and  $Sol$,  we elaborate arguments in  this  text  which  include the  restrictive  conditions  for  the  existence of  lattices  in solvable  groups,  as  well as  an algebraically rigid  classification of  discrete subgroups  of isometries  of nilmanifolds. This  is  the  most  original  contribution  among  the  proofs presented  in this  note. 

For  the  geometry $\widetilde{SL_{2}(\mathbb{R}}$, which  is  given as  a  non-trivial  central  extension of $PSL_{2}(\mathbb{R})$ by $\mathbb{Z}$, the arguments  include  an analysis  of  the  behaviour  of the  fixed  point  set  of  discrete subgroups of  isometries  of  the  visual compactification of  the  hyperbolic  plane. 

The  products $S^{2}\times \mathbb{R}$  and  $\mathbb{H}^{2}\times \mathbb{R}$  exhibit  differences  in the  main  argument. For the  latter  the  projection  to  the  hyperbolic  factor  is  analyzed,  and  the  result  is  reduced to  the  observation  that  the  isometry groups  of  a two  dimensional hyperbolic  orbifold  are  finite, which  are  combined together  with the  fact  that discrete subgroups of  isometries   on $\mathbb{H}^{2}\times \mathbb{R}$ project  to  isometry  groups of  $\mathbb{H}^{2}$ producing finite  volume  orbifolds.

 For  the former, $S^{2}\times \mathbb{R}$, the    key  remark is  that  a  discrete cocompact  isometry  group  can  be  realized  as  subgroup  of  $SO(3)\times  S^{1}$.  This  result  is  followed  from Tollefsons' classification  results \cite{tollefson} of  the  groups  acting  on three manifolds  with  that geometry.

 \subsection{Overview of Classification Results}
  
  We  notice  that  the   classification  of isometry groups  presented  here  has as  consequence   classification results   for  finite group actions  on  three  manifolds.  There  are  three general kinds of  behaviour: 
  \begin{itemize}
  \item Within  the geometries   whose  isometry groups   are  extensions  involving a  discrete and  finite  volume subgroup  of $SL_{2}(\mathbb{R})$,  any  finite  group  can  act by  isometries.  This  concerns  the  geometries  $\mathbb{H}^{3}$,  $\widetilde{SL_{2}}$ and  $\mathbb{H}^{2}\times  \mathbb{R}$.   
  \item For  the  case  of spherical  factors and  the euclidean case,  the  classical results by  Tollefson, Seifert-Threlfall  and  Bieberbach  theorems exhaust the  class of  finite  groups acting by  isometries. This  is  recorded  in \cite{mccullough}, \cite{tollefson}, \cite{nowacki}. There  exists  a  classification  of  free (topological) finite  group  actions  in \cite{leeshinyokura}, \cite{hajokimlee}.
  
  \item For Sol,  any  cyclic  group  can  be  realized   as  a  consequence  of  the discussion  of  example \ref{ex:sol-lattice}. For the  case  Nil,   \ref{prop:nil-discrete-proj} gives  a  classification depending  on the  results  developed  in this  note.  
 
  Notice  that  there  exists  a  classification  of  topological free  actions  of  finite  groups on Nil and  Sol manifolds based  on  the  $p$- rank  and P.A. Smith theory \cite{LeeJoo}, \cite{leeshinyokura}. 
  
  \end{itemize}
  
\subsection{Proof  Schema for  Theorem \ref{zimmer:3man}.}
  In this  short  subsection we  summarize the  detailed chain  of  implications leading  to the  proof  of  Theorem  \ref{zimmer:3man}.

{ \bf Euclidean Geometry}

According  to the  Bieberbach Theorems,  there  exists a discrete free  abelian  subgroup  of  translations $T$, which  has  rank less  or  equal to three.  The  assertion  of  theorem \ref{zimmer:3man}  for  the  isometry  group  of  $R^{3} / \Gamma$ will be  verified by  examining  the  rank  of  the  translation  subgroup  $T$, and  discarding rank two and  one  by  producing  a  quotient  of  infinite  volume. 

{\bf Nil Geometry.}

For  the Nil geometry, Theorem \ref{zimmer:3man} is  a  consequence of \ref{lema:nil-infinite-volume}, characterizing  lattices  of infinite  volume, and Proposition \ref{prop:nil-discrete-proj} giving  an exact  sequence between isometry groups.

{\bf Spherical Geometry.}

The  theorem  is  a  direct  consequence  of  the  classification  of finite  group  actions  on three -manifolds, as well  as  the  determination  of  the components  of the  isometry  groups \ref{theorem:homotopytypespherical}. Notice  that  there  are  neither   non-discrete  subgroups  of isometries nor groups  whose  quotient shows  infinite  volume within the spherical geometry. 

{\bf $S^2\times \mathbb{R}$ Geometry.}

The  theorem  is consequence  of  the  splitting  of  the isometry  groups, as well  as  the  characterization  of  discrete subgroups  of  isometries in \ref{lema:discrete_subgroups_S2_R}, finally  concluding  in \ref{Thm:general_isometry_gropu_S2_R}.

{\bf Sol Geometry.}

The  theorem  is  stated  as Corollary \ref{cor:solv-isometry-group}, which  depends  on the  determination of  centralizers  in \ref{lema:solv-centralizer}, and  the  determination  of  finite  volume  in \ref{prop:solv-discrete-projection}.

{\bf Hyperbolic Geometry.}

The  theorem  is  direct  consequence  of  Lemma \ref{lema:hyperbolic-isometry-group},  which  is  in  turn  consequence  of Borel Density, or  the  preceding  argumentation there. See also \cite{bochner}.

{\bf $\mathbb{H}^2\times \mathbb{R}$ Geometry.}
The  theorem  is  stated  in  \ref{teo:hyperbolic-fibration-isometry-group},  and  it   is  consequence  of  Theorem\ref{teo:isometry-group-product}, stating  the  splitting  of  isometry  groups  of  the factors.

{\bf $\widetilde{SL_{2}(\mathbb{R})}$ Geometry. }

The  Theorem is  also  stated  in \ref{teo:hyperbolic-fibration-isometry-group}, and it  is  consequence  of Lemma \ref{lema:hyperbolic-isometry-group}, and  Proposition \ref{prop:hyperbolic-discrete-projections}.

\section{Euclidean Geometry}\label{section:euclidean}
Recall  that  a three  manifold  is  Euclidean  if  it  is  locally  isometric  to  the  Euclidean three  dimensional  space  $\mathbb{R}^3$. The  isometry  group  of  the  three  dimensional  space  is  the  semidirect  product $E(3)= \mathbb{R}^{3}\rtimes  O(3).$

Let  $\Gamma $  be  a  discrete subgroup  of isometries $E(3)$. It  is  a consequence  of   the  Bieberbach  Theorems,  as  interpreted  by Nowacki \cite{nowacki},  that  there  exists  a  free  abelian group $T$ of rank  $\leq 3$  and having finite index in $\Gamma$.  

\begin{proof}[End  of  proof  of  Theorem  \ref{zimmer:3man} for  euclidean  geometry]
We  will  now  verify  the  assertion  of  theorem \ref{zimmer:3man}  for  the  isometry  group  of  $R^{3} / \Gamma$ by  examining  the  rank  of  the  translation  subgroup  $T$.  
\begin{itemize}
\item If  the  rank  of  $T$ is one, then $\Gamma$  is  a  finite  extension  of $\mathbb{Z}$, and $\mathbb{R}^{3}/T$ is   either  the   interior  of  a solid  torus  or  the topological interior  of  a   solid Klein  Bottle,  depending  on   the  orientability, where  the  generator  of  $T$  acts  as a screwdriver  isometry (combination  of  a rotation  around  an axis  and  a translation along a parallel  direction). It  follows  that  $   \mathbb{R}^{3}/T$  has  infinite  volume. 

\item If the  rank  of  $T$ is  two, then  $ \mathbb{R}^{3}/T$  is  the  total  space  of  a   line  bundle  over   either the  torus  or  the  Klein  bottle,  and  $  \mathbb{R}^{3}/T$ has  infinite  volume. 

\item  If  the  rank  of  $T$  is  three, then  the  isometry  group  of  $E(3)/ \Gamma$  is  a  finite   extension of a  rank  three torus by  a  finite  subgroup. 
\end{itemize}
\end{proof}
\subsection{Classification}
The  classification  of  (topological) finite  group  actions  on  the  torus by  isometries  has  been  concluded  by  work of Lee, Shin and  Yokura \cite{leeshinyokura} and Ha, Jo, Kim  and  Lee \cite{hajokimlee}.
 
It  follows  from the  Bieberbach  theorems  that  any  topological  action  on the  three  torus  is topologically conjugated  to  an  isometry;  moreover,  by the  fact  that  the  three  dimensional torus  is sufficiently  large in  the  sense  of  Heil  and  Waldhausen, \cite{waldhausenlarge}, any  homotopy  equivalence  is  homotopic  to a  homeomorphism, and  any two homotopic  homeomorphisms are isotopic. 
 
 {\bf Connected  components}

The  isometry  groups  of  co-compact euclidean  orbifolds  have  been  determined  by  Ratcliffe  and Tschantz    \cite{ratcliffetschantz}, in  Theorem  1  and  Corollaries  1   and 2  in pages  46  and  47,  which  we state   now  for  later  reference. 

\begin{theorem}\label{orbifold:euclidean}
The isometry group  of a  cocompact  euclidean orbifold $\mathbb{R}^{3}/\Gamma$ is  a  compact  Lie  group  whose identity  component  is  a  Torus  of  dimension equal  the  first  Betti number  of  the  group $\Gamma$, which corresponds to the rank of the abelian group $\Gamma / [\Gamma, \Gamma]$.
\end{theorem}

\subsection{Examples}

To understand why a compact quotient $\R^3 / \Gamma$ could have as isometry group a torus of smaller dimension than $3$, we can take a look at two examples in dimension two:
\begin{example}\label{remark:general_decreace_symmetries1}
The group $\Z^2$ is a discrete subgroup of $Iso(\R^2)$, such that the quotient $\R^2 / \Z^2$, has the torus $N_{\R^2}(\Z^2) / \Z^2 \cong \R^2 / \Z^2$ acting naturally by isometries, however the full isometry group $Iso(\R^2 / \Z^2) \cong (\R^2 / \Z^2) \rtimes D_4$ is bigger.
\end{example}

\begin{example}\label{remark:general_decreace_symmetries}
We may extend the previous example to the group $\Lambda = \Z^2 \rtimes D_4$, which is a discrete subgroup of $Iso(\R^2)$, such that it is not completely contained in $\R^2$ and produces a compact quotient $\R^2 / \Lambda$, homeomorphic to the $2$-sphere $S^2$. To compute the isometry group, we observe the contentions
	\[	N_{\R^2} (\Lambda) = \{(n/2, n/2 + m) : n,m \in \Z \} \subset N_{\R^2} (\Z^2) = \R^2,	\]
and $N_{Iso(\R^2)}(\Lambda ) = Aut(\Z^2) \rtimes N_{\R^2} (\Lambda) \cong D_4 \rtimes N_{\R^2} (\Lambda)$, which gives us
	\[	N_{Iso(\R^2)}(\Lambda ) / \Lambda \cong (D_4 \rtimes N_{\R^2} (\Lambda)	) / (D_4 \times \Z^2) \cong \Z_2.	\]
This gives us a finite isometry group $Iso(\R^2 / \Lambda) \cong \Z_2$. Observe that if $\sigma \in D_4$ and $v \in \Z^2$, then the commutator of these elements is $[\sigma,v] = \sigma(v) - v \in \Z^2$ and we can see that the commutator group $[\Gamma , \Gamma]$ contains a lattice subgroup of $\R^2$ which implies that $\Gamma / [\Gamma, \Gamma]$ is finite, verifying Theorem \ref{orbifold:euclidean}.
\end{example}

\section{Nil geometry}

\subsection{Riemannian geometry of the Heisenberg group}

If $\mathbb{F}$ is a commutative ring, denote by $H_\mathbb{F}$ the group of $3 \times 3$ upper triangular matrices over $\mathbb{F}$ with $1$ in the diagonal, that is
	\[	H_\mathbb{F} = \left\{\left(\begin{array}{ccc} 1 & x & z \\ 0 & 1 & y \\ 0 & 0 & 1 \end{array}\right) : x,y,z \in \mathbb{F}	\right\}.	\]
The group $H_\R$ is a Lie group called the three dimensional Heisenberg group that fits  into the exact sequence
	\[	1 \rightarrow \R \rightarrow H_\R \rightarrow \R^2 \rightarrow 1,	\]
where $\R \subset H_\R$ is its center. The three matrices
	\[	e_1 = \left(\begin{array}{ccc} 0 & 1 & 0 \\ 0 & 0 & 0 \\ 0 & 0 & 0 \end{array}\right), \quad
		e_2 = \left(\begin{array}{ccc} 0 & 0 & 0 \\ 0 & 0 & 1 \\ 0 & 0 & 0 \end{array}\right), \quad
		e_3 = \left(\begin{array}{ccc} 0 & 0 & 1 \\ 0 & 0 & 0 \\ 0 & 0 & 0 \end{array}\right);
		\]
determine a canonical basis of the tangent space at the identity $T_I (H_\R)$, so that its translations by left-multiplications gives us a basis of left invariant vector fields denoted by $X_j$ with $X_j(I) = e_j$. For a fixed element  
	\[	g = \left(\begin{array}{ccc} 1 & x & z \\ 0 & 1 & y \\ 0 & 0 & 1 \end{array}\right) \in H_\R, 	\]
the vector fields at $T_g (H_\R)$ have expresions
	\[	X_1(g) = \left(\begin{array}{ccc} 0 & 1 & 0 \\ 0 & 0 & 0 \\ 0 & 0 & 0 \end{array}\right), \quad
		X_2(g) = \left(\begin{array}{ccc} 0 & 0 & x \\ 0 & 0 & 1 \\ 0 & 0 & 0 \end{array}\right), \quad
		X_3(g) = \left(\begin{array}{ccc} 0 & 0 & 1 \\ 0 & 0 & 0 \\ 0 & 0 & 0 \end{array}\right).
		\]
If we consider the global coordinates
	\[	\R^3 \rightarrow H_\R, \qquad (x,y,z) \mapsto \left(\begin{array}{ccc} 1 & x & z \\ 0 & 1 & y \\ 0 & 0 & 1 \end{array}\right),	\]
then a vector $v \in T_g (H_\R)$ decomposes as
	\[	v = v_1 e_1 + v_2 e_2 + v_3 e_3 = v_1 X_1(g) + v_2 X_2(g) + (v_3 - x v_2) X_3(g),	\]
so that the left-invariant metric in $H_\R$ having $X_j(g)$ as an orthonormal basis is given in this coordinates as $ds^2 = dx^2 + dy^2 + (dz - x dy)^2$. Being left-invariant, this metric has $H_\R$ as a subgroup of isometries given by left multiplication
	\[	L_g : H_\R \rightarrow H_\R, \qquad L_g(h) =gh,	\]
for every $g \in H_\R$.  Notice that there are other isometries that don't come from left multiplication of $H_\R$.  Such isometries form a group isometric to the orthogonal group $O(2)$ generated by the reflection $R(x,y,z) = (x,-y,-z)$ and the twisted rotations
	\[ 	m: S^1 \times H_\R \rightarrow H_\R, \qquad m_\theta(x,y,z) = (\rho_\theta(x,y), z + \eta_\theta(x,y)), \]
where $\rho_\theta$ is a rotation in the $(x,y)$-plane with angle $\theta$, $\eta_\theta$ is a polynomial function in $x$ and $y$ and trigonometric in $\theta$. 

The full isometry group $Iso(H_\R)$ can be described  as a semi-direct product $  H_\R \rtimes O(2)$, because
	\[	m_\theta \circ L_g \circ m_\theta^{-1} = L_{m_\theta(g)}, \qquad R \circ L_g \circ R = L_{R(g)},	\]
meaning  that   the exact exact sequence
	\[	1 \rightarrow \R \rightarrow Iso(H_\R) \rightarrow Iso(\R^2) \rightarrow 1,		\]
induced by the action on the quotient by the center $H_\R /\R \cong \R^2$ splits off, see \cite{Scott} for more details. 

\subsection{Examples}
In this section we describe a series of ilustrative examples that capture the behaviour of every discrete subgroup of $Iso(H_\R)$.

\begin{example}\label{ex:square_nil_lattice}
The group $H_\Z \subset H_\R$ is a discrete subgroup so that the exact sequence determining $H_\R$ induces the fiber-bundle structure
	\[	\R /\Z \rightarrow H_\R / H_\Z \rightarrow \R^2 / \Z^2	\]
and thus $H_\Z$ is a lattice subgroup of $H_\R$ such that $H_\R / H_\Z$ is a compact Riemannian manifold. As the conjugation can be computed as
	\[	g = (x,y,z), \qquad g \ (n,m,p) \ g^{-1} = (n, m, p + xm - yn),	\]
the normalizer in $H_\R$ is $N_{H_\R}(H_\Z) = \left\{	\left( n ,m, p \right) : n,m \in \Z, \ p \in \R  \right\}$. This gives us the isometries in the quotient
	\[	S^1 \cong N_{H_\R}(H_\Z)/H_\Z \hookrightarrow Iso(H_\R / H_\Z). 	\]
 
We consider now  the  normalizer of  the  Heisenberg  group  in $Iso(H_\R)$.  This can be  determined as
	\[	N_{Iso(H_\R)}(H_\Z) = \left\{	\left( n ,m, p \right) : n,m \in \Z, \ p \in \R  \right\} \rtimes D_4,	\]
where the Dihedral group $D_4$ is generated by the isometries
	\[	m_{\pi/2}(n,m,p) = (-m,n,p -nm), \qquad R(n,m,p) = (n,-m,-p),	\]
so that what we get is $Iso(H_\R/H_\Z) \cong S^1 \rtimes D_4$. 

We can modify this example by adding the dihedral group to the lattice, so that we have the fiber bundle structure
	\[	\R / \Z \rightarrow H_\R / (H_\Z \rtimes D_4) \rightarrow \R^2 / (\Z^2 \rtimes D_4) \cong S^2. 	\]
The discussion on Example \ref{remark:general_decreace_symmetries} explains the  last quotient in the sequence). Notice  that  we decreased the normalizer 
	\[	N_{Iso(H_\R)}(H_\Z \rtimes D_4) = \{(n,m,l) : n,m,2l \in \Z \} \rtimes D_4,	\]
so that $Iso(H_\R/(H_\Z \rtimes D_4)) \cong \Z_2$.
\end{example}

\begin{example}\label{ex:square_p_nil_lattice}
Fix a positive integer $p \in \mathbb{N}$ and consider the lattice 
	\[	G_p = \left\{	\left( n,m , \frac{l}{p} \right) : n,m,l \in \Z \right\} \subset H_\R,	\]
which has as normalizer group in $H_\R$ the group
		\[	N_{H_\R}(G_p) = \left\{	\left( \frac{n}{p}, \frac{m}{p} , r \right) : n,m \in \Z, \ r \in \R \right\},	\]
and normalizer group in $Iso(H_\R)$, the group $N_{H_\R}(G_p) \rtimes D_4$, with Dihedral group $D_4 = \langle m_{\pi/2}, R \rangle$ as before. The isometry group is characterized by the exact sequence
	\[	1 \rightarrow S^1 \rightarrow Iso(H_\R / G_p) \rightarrow D_4 \ltimes (\Z_p \times \Z_p) \rightarrow 1	\]
and we recover the previous example by taking $p = 1$.
\end{example}

\begin{example}\label{ex:hexagonal_nil_lattice}
Fix a positive integer $p \in \mathbb{N}$ and consider the lattice
	\[	L_p = \left\{	\left( \frac{n}{2} + m , \frac{\sqrt{3} n}{2} , \frac{\sqrt{3} l}{2 p} \right) : n,m,l \in \Z \right\} \subset H_\R,	\]
so that it has normalizer group in $H_\R$
	\[	N_{H_\R}(L_p) = \left\{	\left( \frac{n}{2p} + \frac{m}{p} , \frac{\sqrt{3} n}{2 p} , r \right): n,m \in \Z, \ r \in \R \right\}.	\]
As the group $L_p$ projects to a hexagonal lattice in $\R^2$, we should expect to have a Dihedral group $D_6$ normalizing $L_p$, however, the rotation $m_{\pi/3} : H_\R \rightarrow H_\R$ given by 
	\[	m_{\pi/3}(x,y,z) = \left(\frac{1}{2} \left( x - \sqrt{3}y \right), \frac{1}{2} \left( y + \sqrt{3}x \right), z + \frac{\sqrt{3}}{8}\left( y^2 - x^2 - 2 \sqrt{3} xy \right) \right),	\]
doesn't preserve $L_p$. To fix this, we must add a translation mixed with the rotation.Put $g = \left( \frac{1}{8}, -\frac{ \sqrt{3} }{8}, 0 \right) \in H_\R$, then $\varphi = m_{\pi/3} \circ L_g \in N_{Iso(H_\R)}(L_p)$, which can be verified using the relation $\varphi \circ L_h \circ \varphi^{-1} = L_{m_{\pi/3}(ghg^{-1})}$.	
We can describe the normalizer group of $L_p$ in $Iso(H_\R)$ in terms  of  generators as
	\[	N_{Iso(H_\R)}(L_p) = \langle L_g, \varphi, R : g \in N_{H_\R}(L_p) \rangle,	\]
where $R(x,y,z) = (x,-y,-z)$, and so, we have the isometry group
	\[	1 \rightarrow S^1 \rightarrow Iso(H_\R / L_p) \rightarrow (\Z_p \times \Z_p) \rtimes D_6  \rightarrow 1,	\]
where the dihedral group $D_6$ is generated by $\langle \varphi, R \rangle$. 
\end{example}

\begin{example}\label{ex:gen_nil_lattice}
All the previous examples can be generalized as follows: Fix $p \in \mathbb{N}$ and $u,v \in \R^2$ linearly independent, so that $\Gamma = \{ nu + mv : n,m \in \Z\} \subset \R^2$ is a lattice. If $(u,0) \times (v,0) = (0,0,\lambda) \in \R^3$, then the group
	\[	M_p = \left\{ \left( n u + mv , \frac{\lambda}{p} l \right) : n,m,l \in \Z \right\} \subset H_\R	\]
is a lattice having normalizer group in $H_\R$
	\[	N_{H_\R}(M_p) = \left\{	\left( \frac{n}{p} u + \frac{m}{p} v , r \right): n,m \in \Z, \ r \in \R \right\}.	\]
The lattice $\Gamma$ has an automorphism group $Aut(\Gamma) \in \{0, \Z_2, D_4, D_6\}$, which is, if non-trivial, generated by a rotation with angle $\theta$ and a reflection. The whole normalizer group is  given in terms of  generators as
	\[	N_{Iso(H_\R)}(M_p) = \langle L_g, \varphi, R : g \in N_{H_\R}(M_p) \rangle,	\]
where $\varphi = m_\theta \circ L_w$. Here, $w = (w_0,0) \in H_\R$ must be chosen so that if
	\[	w \times (u,0) = (0,r_1), \qquad w \times (v,0) = (0,r_2),	\]
then $m_\theta(u, r_1), m_\theta(v, r_2) \in M_p$. Thus, we have an isometry  group of the quotient given by the exact sequence
	\[	1 \rightarrow S^1 \rightarrow Iso(H_\R / M_p) \rightarrow (\Z_p \times \Z_p)  \rtimes Aut(\Gamma) \rightarrow 1.	\]
\end{example}

\begin{remark}
The previous examples give us the general strategy to compute the isometry group of a quotient $H_\R / G$, for $G \subset Iso(H_\R)$ a discrete group of isometries. This strategy is as follows: $G$ projects to a discrete subgroup $\Gamma \subset Iso(\R^2)$ which has a finite index subgroup $\Gamma_0 \subset \R^2$, corresponding to a finite index subgroup $G_0 = G \cap H_\R \subset G$ and a lattice in $H_\R$. The normalizer of $G_0$ projects again a lattice in $\R^2$ and thus $Iso(H_\R /G_0)$ is an extension of a finite group $\Z_p \times \Z_p$ by $S^1$. The isometry group $Iso(H_\R /G)$ is just the previous group with an extra finite group of isometries, coming from the automorphisms of the lattice $Aut(\Gamma)$. This strategy fails if the projection to $\R^2$ is non-discrete, a possibility shown in the following two examples, however, in the case where the quotient $H_\R / G$ has finite volume, we will see that this patological behaviour doesn't occur.
\end{remark}

Here we add two examples of discrete groups whose projected action onto $\R^2$ is non-discrete, these examples capture the general behaviour of discrete groups having this property as we will see in the next section.

\begin{example}
Consider $\varphi: \IN \rightarrow S^1$, a homomorphism with dense image and $g = (0,0,1) \in H_\R$ a generator of the center, so that the group
	\[	\{ (g^{n} , \varphi(n)) : n \in \IN \} \subset  H_\R \rtimes S^1 \cong Iso(H_\R)	\]
is a discrete subgroup of isometries of $H_\R$ with dense projection onto $S^1 \cong SO(2)$ and in particular, with a non-discrete action on $\R^2$. In this example, the projected group leaves fixed the point $p = \frac{1}{1 - \lambda} \in \IC \cong \IR^2$, where $\lambda = \varphi(1)$ and in particular, it is a group of rotations around such point.
\end{example}

\begin{example}
Given a scaling $0<\varepsilon<1$, consider the group generated by $(1,0,1), (\varepsilon, 0, 1) \in H_\R \subset Iso(H_\R)$ and $-1 \in \IS^1 \subset Iso(H_\R)$. This is a discrete subgroup of $Iso(H_\R)$, which projects to a non-discrete subgroup of $Iso(\IR^2)$ leaving fixed the line $\{(x,0) : x \in \IR \} \subset \IR^2$.
\end{example}

\begin{remark}
The most symmetric lattices in $\R^2$ are the square and hexagonal lattices, having linear symmetry groups $D_4$ and $D_6$. Theorem \ref{prop:nil-isometry-group} tells us that the generalizations of these lattices to $H_\R$, described in Example \ref{ex:square_p_nil_lattice} and Example \ref{ex:hexagonal_nil_lattice} are the most symmetric finite volume quotients $H_\R / G$, with isometry groups
	\[	1 \rightarrow S^1 \rightarrow Iso(H_\R/G) \rightarrow  (\Z_n \times \Z_n) \rtimes  D  \rightarrow 1,	\]
with $D$ equal to $D_4$ and $D_6$ respectively, and $n \in \mathbb{N}$.
\end{remark}

\subsection{Classification of discrete subgroups of isometries}

In this section $H_\R$ denotes the Heisenberg Lie group considered as a Riemannian manifold with respect to the left-invariant metric constructed in the previous section. Here, we describe the conditions on which a discrete group on $Iso(H_\R)$ induces a discrete action on the Euclidean plane $\R^2$.

\begin{proposition}\label{prop:nil-discrete-proj}
If $G$ is a discrete subgroup of isometries of $H_\R$, then the exact sequence
	\[	1 \rightarrow \R \rightarrow Iso(H_\R) \rightarrow Iso(\R^2) \rightarrow 1	\]
induces an exact sequence
	\[	1 \rightarrow K \rightarrow G \rightarrow \Gamma \rightarrow 1,	\]
where $\Gamma \subset Iso(\R^2)$ is either discrete or it is an abelian group leaving fixed either a point or a line. Moreover, 
	\begin{enumerate}
		\item if $\Gamma \subset Iso(\R^2)$ has a finite index lattice, then $K \subset \R$ is a non-trivial discrete subgroup and 

		\item if $\Gamma$ is non-discrete and leaves fixed a line, then there is a finite index subgroup of $G$ which is contained in $H_\R$.
	\end{enumerate}
\end{proposition}

\begin{proof}
Observe first that $K = G \cap \R$ is a discrete subgroup of isometries of $\R$ and so, if non-trivial, there is an isomorphism $\R / K \cong S^1$. The exact sequence
	\[	1 \rightarrow S^1 \rightarrow H_\R/K \rightarrow \R^2 \rightarrow 1	\]
gives us
	\[	1 \rightarrow S^1 \rightarrow Iso(H_\R)/K \rightarrow Iso(\R^2) \rightarrow 1,	\]
which has compact Kernel and thus, any discrete group in $Iso(H_\R)/K$ projects to a discrete group in $Iso(\R^2)$. This argument tells us that if $K$ is non-trivial then $\Gamma$ is discrete in $Iso(\R^2)$, because it is the projection of $G/K$ with compact kernel, and $G/K$ is always discrete in $Iso(H_\R)/K$. Suppose from now on that $K$ is trivial. If we identify $\R^2 \cong \C$ as a Euclidean space, then we can realize the group of orientation preserving isometries of the plane $\R^2$ as the matrix group
	\[	Iso^+(\R^2) \cong \left\{	\left(\begin{array}{cc} \lambda & z \\ 0 & 1 \end{array}\right)	: \lambda, z \in \C, \ |\lambda| = 1 \right\}		\]
with action
	\[	\left(\begin{array}{cc} \lambda & z \\ 0 & 1 \end{array}\right) \left(\begin{array}{c} w \\ 1 \end{array}\right) = \left(\begin{array}{c} \lambda w + z \\ 1 \end{array}\right).	\]
Observe that the restriction $Iso^+(\R^2) \subset Iso(\R^2)$ reduces the discusion to a subgroup of index 2, which doesn't alter the property of discreteness. We recall two important properties on commutators. First, commutators of two isometries give elements of pure translation part
	\[	\left[ \left(\begin{array}{cc} \lambda & z \\ 0 & 1 \end{array}\right) , \left(\begin{array}{cc} \beta & w \\ 0 & 1 \end{array}\right) \right] = \left(\begin{array}{cc} 1 & (z-w) + (\lambda w - \beta z) \\ 0 & 1 \end{array}\right),	\]
which tells us that $[G,G]$ projects to a subgroup of $Iso(\R^2)$ with only translation part, and so $[G,G] \subset H_\R$. Second, the commutator in $H_\R$ satisfies the relation
	\[	\left[	\left(\begin{array}{ccc} 1 & x & r \\ 0 & 1 & y \\ 0 & 0 & 1 \end{array}\right) , \left(\begin{array}{ccc} 1 & u & s \\ 0 & 1 & v \\ 0 & 0 & 1 \end{array}\right) \right] = \left(\begin{array}{ccc} 1 & 0 & x v - uy \\ 0 & 1 & 0 \\ 0 & 0 & 1 \end{array}\right),	\]
which has the geometric interpretation: if two elements of $H_\R$ project to the vectors $(x,y)$ and $(u,v)$, then its commutator is an element of the center $\R = Z(H_\R)$ whose magnitud is the area of the projected vectors. As we are under the supposition that $G \cap \R$ is trivial, the two previous relations on commutators tells us that $[G,G]$ is a commutative group and the corresponding projected group satisfies
	\[	[\Gamma,\Gamma]	\subset \left\{ \left(\begin{array}{cc} 1 & r z_0 \\ 0 & 1 \end{array}\right) : r \in \R \right\}	\]
for some $z_0 \in \C$. Suppose first that $\Gamma$ is non-commutative. The commutation relation
	\[	\left[\left(\begin{array}{cc} \lambda & z \\ 0 & 1 \end{array}\right), \left(\begin{array}{cc} 1 & z_0 \\ 0 & 1 \end{array}\right) \right] = \left(\begin{array}{cc} 1 & (\lambda-1)z_0 \\ 0 & 1 \end{array}\right),	\]
and the hypothesis that all the translation elements of $[\Gamma, \Gamma]$ are linearly dependent give us the condition $r = (1-\lambda)$ for some $r \in \R$ and as $|\lambda| = 1$, the only options are $\lambda = \pm 1$. As $\Gamma$ is non-commutative, there is at least one element that is not a translation, that is
	\[	\left(\begin{array}{cc} -1 & z \\ 0 & 1 \end{array}\right) \in \Gamma,	\]
and without loss of generality, we can change $\Gamma$ by $h \Gamma h^{-1}$ (where $h$ is the translation by $1/2 z$) so that in fact
	\[	\left(\begin{array}{cc} -1 & 0 \\ 0 & 1 \end{array}\right) \in \Gamma,	\]
this conjugation leaves $[\Gamma, \Gamma]$ invariant. Observe also that
	\[	\left[\left(\begin{array}{cc} \beta & w \\ 0 & 1 \end{array}\right), \left(\begin{array}{cc} -1 & 0 \\ 0 & 1 \end{array}\right) \right] = \left(\begin{array}{cc} 1 & 2 w \\ 0 & 1 \end{array}\right) \subset [\Gamma, \Gamma],	\]
implies by the same argument that $\beta = \pm 1$ and $w = s z_0$ for some $s \in \R$, and thus $\Gamma$ preserves the line generated by $z_0$. If on the other hand $\Gamma$ is commutative and contains an element of the form
		 \[	\left(\begin{array}{cc} \lambda & z \\ 0 & 1 \end{array}\right), \qquad \lambda \neq 1,	\]
this element has as a unique fixed point $\frac{-z}{\lambda -1}$. As $\Gamma$ is a commutative group, every element of $\Gamma$ must fix $\frac{-z}{\lambda - 1}$, and thus, it consists of rotations around this point. If no such element exists, $\Gamma$ consists of elements with purely translation part, which tells us that $G \subset H_\R$. We observe that in this last case, two elements $a,b \in G$ which project to two linearly independent vectors in $\Gamma$ must satisfy that $e \neq [a,b] \in G \cap K$, which can't happen by hypothesis, so $\Gamma$ is a subgroup of the group $\{r \omega : r \in \R \}$ for some $\omega \in \C$ and thus $\Gamma$ preserves the line generated by $\omega$.
\end{proof}

\begin{lemma}\label{lema:nil-infinite-volume}
Let $G$ be a discrete subgroup of isometries of $H_\R$ together with the projection to the isometry group of $\R^2$
	\[	G \rightarrow \Gamma \subset Iso(\R^2).	\]
If $\Gamma$ preserves either a line or a point in $\R^2$, then the orbifold $H_\R/G$ has infinite volume.
\end{lemma}

\begin{proof}
Suppose first that $\Gamma$ preserves the line $\R v$, then as a consequence of either Bieberbach's Theorem if $\Gamma$ is discrete, or as a consequence of the proof of Proposition \ref{prop:nil-discrete-proj} if $\Gamma$ is non-discrete, $G$ has a finite index subgroup that is contained in $H_\R$. Passing to a finite index subgroup doesn't change the property of having finite co-volume so without loss of generality we may suppose that $G \subset H_\R$. There is a fundamental domain that has non-empty interior, given for example by the Dirichlet's fundamental domain $\{q \in H_\R : d(q_0,q) < d(q_0,\gamma(q) ), \ \gamma \in G \setminus \{e\} \}$, with respect to the Riemannian distance $d$, see \cite{Rat}. In particular there is a subset of the form
	\[	D = \{ (t v + s v^\perp , \lambda r_0 ) \subset \C \times \R : (s,t,\lambda) \in (-\varepsilon,\varepsilon)^3 + (s_0,t_0,\lambda_0) \} \subset H_\R	\]
such that no two elements of $D$ can be identified with an element of $G$. As $\Gamma$ preserves the line $\R v$, then we can see that no two elements of $\widetilde{D}$ can be identified with an element of $G$, where
	\[	\widetilde{D} = \{ (t v + s v^\perp , \lambda r_0 ) \subset \C \times \R : (s,t,\lambda) \in \R \times (-\varepsilon,\varepsilon)^2 + (0,t_0,\lambda_0) \}	\]
but $\widetilde{D} = \bigcup_j D_j$, where
	\[	D_j = \{ (t v + s v^\perp , \lambda r_0 ) \subset \C \times \R : (s,t,\lambda) \in (-\varepsilon,\varepsilon)^3 + (s_j,t_0,\lambda_0) \} \subset H_\R	\]
and every $D_j$ can be obtained by translating $D$ with an element of $H_\R$, thus
	\[	Vol(H_\R/G) \geq Vol(\widetilde{D}) = \sum_j Vol(D_j) = \sum_j Vol(D) = \infty.	\]
The second possibility is when $\Gamma$ is a commutative group preserving a point, that is, $\Gamma$ is conjugated to a subgroup of $SO(2)$. Again there is a fundamental domain of $G$ with non-empty interior and in particular, there is a subset
	\[	\Omega = \{ (r e^{i \theta}, sr_0) \subset \C \times \R : (r,\theta, s) \in (-\varepsilon, \varepsilon)^3 + (a,b,c) \} \subset H_\R,	\]
such that no two elements of $\Omega$ can be identified with an element of $G$. As $\Gamma$ acts only as rotations in the $\C$ plane, we can enlarge as before $\Omega$ to the subset
	\[	\widetilde{\Omega} = \{ (r e^{i \theta}, sr_0) \subset \C \times \R : (r,\theta, s) \in \R_{> 0} \times (-\varepsilon, \varepsilon)^2 + (0,b,c) \},	\]
so that no two elements of $\widetilde{\Omega}$ can be identified with an element of $G$. As before, we have a countable union of disjoints sets contained in $\widetilde{\Omega}$ that are translated copies of $\Omega$, that is 
	\[	\bigcup_i \Omega + (\omega_j,0)  \subset \widetilde{\Omega}	\]
and $Vol(H_\R/G) \geq Vol(\widetilde{\Omega}) \geq \sum_j  Vol(\Omega + (\omega_j,0))	= \sum_j Vol(\Omega) = \infty$.
\end{proof}

\begin{lemma}\label{lema:carac_latt}
If $u,v \in \R^2$ are two linearly independent vectors, with $(u,0) \times (v,0) = (0,0,\lambda) \in \R^3$ and $n \in \mathbb{N}$, $r,s \in \R$, then the group
	\[	G = \left\langle \left(u , r \right), \left(v,s \right), \left(0,0,\frac{\lambda}{n} \right) \right\rangle \subset H_\R	\]
is a lattice in $H_\R$. Conversely, every lattice in $H_\R$ can be obtained like this. 
\end{lemma}

\begin{proof}
Observe that the center of $G$ is the subgroup $K = \left \{ \frac{\lambda p}{n} : p \in \Z \right\}$ and if $(x,y,z),(x,y,z') \in G$, then
	\[	(x,y,z)^{-1} \cdot (x,y,z') = (0,0,z'-z) \in K,	\]
so that for $k,l \in \mathbb{N}$ fixed, and 
	\[	(u,r)^k \cdot (v, s)^l =(ku + lv , r_{k,l}) 	\]
the level set
	\[	\{(w, z) \in G: w = ku +lv \} = \left\{\left(k u + l v, r_{n,m} + \frac{\lambda p}{n}\right) : p \in \Z \right\}	\]
is discrete and thus, $G$ is a discrete subgroup of $H_\R$. If $\Gamma = \left\{ k u + l v : k,l \in \Z \right\}$ denotes the projection of $G$ onto $\R^2$, then there is an exact sequence
	\[	1 \rightarrow K \rightarrow G \rightarrow \Gamma \rightarrow 1	\]
which induces the fiber bundle structure
	\[	S^1 \cong \R /K	\rightarrow H_\R / G \rightarrow \R^2 / \Gamma \cong S^1 \times S^1,	\]
	which tells us that $H_\R/G$ is compact and thus, $G$ is a lattice in $H_\R$. Suppose now that $L \subset H_\R$ is a lattice, then by Lemma \ref{lema:nil-infinite-volume}, $L$ projects to a lattice subgroup of $\R^2$, generated by two linearly independent vectors $u', v' \in \R^2$ such that $(u',0) \times (v',0) = (0,0,\lambda')$, with $0 \neq \lambda' \in \R$ and observe that if $g = (u',r'), h = (v',s') \in L$, then their commutator is $[g,h] = (0,0,\lambda')$. As the intersection $K' = G \cap Z(H_\R)$ is discrete and contains the non-trivial element $(0,0,\lambda') \in K'$, then there is an integer $n' \in \mathbb{N}$ such that $K' =  \left \{ \frac{\lambda' p}{n'} : p \in \Z \right\}$ and thus, the lattice $L$ is generated by the set $\left\{ (u',r'), (v',s'), \left(0,0,\frac{\lambda' }{n'} \right) \right\}$.
\end{proof}
{\bf Connected components }
\begin{theorem}\label{prop:nil-isometry-group}
If $G \subset Iso(H_\R)$ is a discrete subgroup such that $H_\R/G$ has finite volume, then there is an exact sequence
	\[	1 \rightarrow C \rightarrow Iso(H_\R/G) \rightarrow  F  \rightarrow 1,	\]
where $F$ is a finite group, and $C \subset S^1$ is a closed subgroup. In particular, either $Iso(H_\R/G)$ is finite, or it is a finite extension of $S^1$.
\end{theorem}

\begin{proof}
By proposition \ref{prop:nil-discrete-proj} and Lemma \ref{lema:nil-infinite-volume}, the projection of $G$ to $Iso(\R^2)$ has a lattice $\Gamma \subset \R^2$ as a finite index subgroup. This is equivalent to the fact that $L = G \cap H_\R$ is a lattice in $H_\R$ and a finite index subgroup in $G$. By Lemma \ref{lema:carac_latt}, there are $u,v \in \R^2$, $\lambda, r,s \in \R$, with $\lambda \neq 0$, and $n \in \mathbb{N}$, such that $\Gamma = \{ ku + lv : k,l \in \Z \}$ and 
	\[	L = \left\langle \left(u , r \right), \left(v,s \right), \left(0,0,\frac{\lambda}{n} \right) \right\rangle \subset H_\R.	\]
As seen in Example \ref{ex:gen_nil_lattice}, the  group $N_{H_\R}(L) = \left\{	\left( \frac{n}{p} u + \frac{m}{p} v , r \right): n,m \in \Z, \ r \in \R \right\}$ is the normalizer of $L$ in $H_\R$. Denote by $Aut(\Gamma) \subset O(2)$ the subgroup that preserves the lattice $\Gamma$ and observe that an element $\varphi = \sigma \circ L_g \in Iso(H_\R)$ satisfies that $\varphi \circ L_h \circ \varphi^{-1} = L_{\sigma(g h g^{-1})}$. As $\sigma(g h g^{-1})$ and $\sigma(h)$ have the same projection onto $\Gamma$, then if $\varphi$ normalizes $L$, $\sigma \in Aut(\Gamma)$ and we have that
	\[	1 \rightarrow K \rightarrow G \rightarrow F' \ltimes \Gamma \rightarrow 1,	\]
for some subgroup $F' \subset Aut(\Gamma)$ and $K = L \cap \R$. As $H_\R$ is normal in $Iso(H_\R)$, we see that $N_{Iso(H_\R)}(G) \subset N_{H_\R}(G)$, and thus, by applying a trick as in Example \ref{remark:general_decreace_symmetries}, we may describe the greater normalizer as
	\[	1 \rightarrow H \rightarrow N_{Iso(H_\R)}(G) \rightarrow F'' \ltimes \Lambda \rightarrow 1,	\]
with $F'' \subset Aut(\Gamma)$ a finite group and $\Lambda \subset \R^2$ a lattice containing $\Gamma$. Thus, the isometry group is calculated as
	\[	1 \rightarrow C = H/K  \rightarrow N_{Iso(H_\R)}(G)/G \rightarrow F = \left(F'' \ltimes \Lambda \right) / \left(F' \ltimes \Gamma \right) \rightarrow 1,	\]
so that $C$ is either finite, cyclic or $S^1$ and $F$ is finite.
\end{proof}

\section{Spherical Geometry}
This   section is  largely  expository  due  to  the fact  that  the  verification of \ref{zimmer:3man} in the spherical case consists  of the comparison of  the  statement  with the (fundamentally  algebraic)  classification  of groups acting by isometries  on three dimensional spherical manifolds  and  orbifolds.   This  concerns  specifically  the quotient orbifold   of  an  action of  a discrete  group  on  a  spherical three-manifold, that is, a quotient of the form 
	\[	M= S^{3}/ \Gamma ,	\] 
for  $\Gamma$  a finite  subgroup  of $O(4)$. 
The  crucial  point  is  that  the  classification  of  orbifolds  up to orientation preserving  isometry is  equivalent   to  the classification of  subgroups  of $O(4)$. 

The  following  is   a  consequence  of the  classification  of  isometry  groups of spherical  $3$- manifolds  in \cite{mccullough},  tables  2  and  3  in  pages 173 and  176,  relying  on  work   of Mccullough and  collaborators  and  ultimately   going  back  to Seifert, Threlfall, Hopf  and Hattori.  See  \cite{diffeoelliptic}, chapter  1 for  an  account  of  these  facts. 
\subsection{Classification}
\begin{lemma}\label{lemma:isoelliptic}
Up  to  finite  subgroups, the isometry  groups  of   spherical  three  manifolds  are: 
\begin{itemize}
\item $SO(3)$.
\item $O(2)$. 
\item  $O(4)$.
\item  $SO(4)$.
\item $SO(3)$.
\item $ O(2)\times O(2)$
\item $S^{1}\times_{\mathbb{Z}/2} S^{1}$. 
\end{itemize}
In  particular,  these subgroups can be realized as closed  subgroups  of  $O(4)$. 
\end{lemma}

For a complete list of isometry groups of spherical orbifolds, see Chapter 3 of \cite{mecchiaseppi}.

{\bf Connected components}

An important result   by  Hatcher \cite{hatchersmale},  originally  conjectured  by  Smale states  that  the inclusion of  the  isometry group  of $S^{3}$  into  the group  of diffeomorphisms is  a homotopy  equivalence.

The  following  result with   contributions of  many  persons including  (at  least) Asano, Boileau, Bonahon, Birman, Cappell,  Ivanov, Rubinstein, and  Shaneson,    is  a  consequence of  research in  mapping  class  groups and three-  dimensional  spherical  manifolds. It  is  discussed  with  comments  about attribution in  \cite{mccullough}, Theorem  1.1 in  page 3. 
  
\begin{theorem}\label{theorem:pi0smale}
Let  $M$  be  a spherical  manifold, then the  inclusion  of  the  group  of isometries of  $M$  into the  group  of  diffeomorphisms induces a bijection  on path  components. 
\end{theorem}

As of 2022,  the   following   result  in page  2  of  \cite{bamlerkleiner}  is  a  consequence  of  the   study   via Ricci flow  methods of  the   homotopy  type  of  the  spaces  of  positive  scalar  curvature  and 
 the  subspace  of  metrics  which  are  locally  isometric to  either  the  round  sphere $S^{3}$   or  the  round cylinder  $S^{2}\times  \mathbb{R}$. 
 
\begin{theorem}\label{theorem:homotopytypespherical}
Let  $(M,g)$ be  a riemannian  manifold  which  is  an  isometric  quotient  of  the  three  dimensional  round  sphere.  Then, the  inclusion of  the  isometry  group   into the   diffeomorphism  group  is a  homotopy equivalence. 
\end{theorem}

The  following  theorem  was  proved  in  \cite{mecchiaseppi},  using   previous  analysis  of  the  authors  of  Seifert  fibrations  for  spherical  orbifolds. It  is  a  consequence  of tables 2  in page  1302,  table  3  in page 1304  and  table  4  in  page 1308.
  
\begin{theorem}\label{theorem:mecchiaseppi}
Let $X$  be  a  spherical three-manifold, and  let  $\Gamma$  be   a  discrete  group of X. 
\begin{itemize}
\item The  isometry groups of the  orbifold  $X/\Gamma $  are  either  closed  subgroups  of $SO_{4}$  or $PSO_{4}$,  if  the action  is  orientation  preserving.
\item  The  identity  component  of  the isometry  groups  are $S^{1}$, $S^{1}\times  S^{1}$ or  trivial  for  the orientation preserving  case.
\end{itemize}
\end{theorem}
\begin{proof} [ End  of  the proof  of Theorem  \ref{zimmer:3man}  for the  spherical  geometry]  
%Recall  that  an orbifold  quotient of  an  action  of a discrete group  on  a  spherical  three  manifold  is  an  Alexandrov  space  of positive  curvature. It  follows  from  corollary  2.2  in \cite{galazguijarrosurvey} (see Theorem \ref{theo:alexdim3}),  that  any  Alexandrov  space  of  positive  curvature is homeomorphic  to  a spherical  three  manifold   or  suspension  of  $\mathbb{R}P^2$.
T he result thus follows from Lemma \ref{lemma:isoelliptic}.
\end{proof}

\section{ $S^{2}\times \mathbb{R}$ geometry. }

A  three  dimensional  manifold  is   said  to  have  $S^{2}\times  \mathbb{R}$-geometry if  its  universal  covering is  homeomorphic  to $S^{2}\times  \mathbb{R}$. 

The  determination  of the  discrete isometry  groups  of spaces with $S^{2}\times \mathbb{R}$ geometry  is  a  consequence  of  well-known  facts,  which  we  will  gather  here.

We recall the following  result, proved  in 
\cite{kobayashinomizu}, Chapter  VI, Theorem  3.5. 

\begin{theorem}
    
Given  a   product  of riemannian   manifolds $M\times N$  with  $M$ of  constant  sectional  curvature  and  $N$ flat,  the   isometry  group  of $M\times  N$   decomposes  as  a  direct  product $ {\rm Iso }(M)\times {\rm Iso} (N).$ It  follows  that for  a  discrete subgroup $\Gamma\leq {\rm Iso} (S^{2}\times  \mathbb{R}) \cong O(3) \times (\R \times \Z_2)$, the projection onto the second factor $\pi_{\mathbb{R}}(\Gamma) \leq \mathbb{R}  \rtimes \mathbb{Z}/2$ is a discrete  subgroup.
\end{theorem}

We will  use  this    splitting  and  the  classification   of  finite  groups  acting  on  $S^{2}\times \mathbb{R}$, which   was proved  in  by  Tollefson in page 61 of \cite{tollefson},  as  follows: 

\begin{theorem}
There  exist  only  four three-manifolds  covered  by  $S^{2}\times \mathbb{R}$,namely: $S^{2}\times  S^{1}$, the  non  orientable  $S^{2}$- bundle  over  $S^{1}$, $\mathbb{R}P^{1}\times S^{1}$, and $\mathbb{R}P^{3} \#\mathbb{R} P^{2}$. Moreover, the finite  groups  which  act  freely  on $S^{2}\times S^{1}$ are  classified  in  \cite{tollefson}, Corollary  2. They  are: 
\begin{itemize}
\item $\mathbb{Z}/p$, producing  quotient  spaces  homeomorphic  to $S^{2}\times S^{1}$ in the  odd  case, and $\mathbb{R}P^{2}$ in the  even  case as  quotient  space.
\item $\mathbb{Z}/p\times \mathbb{Z}/2$, for $p$ even,  producing  a  quotient  space  homeomorphic   to  $\mathbb{R} P^{2}$, and 
\item $D_{n}$, the  dihedral  group  of  order  $2n$, producing $\mathbb{R} P^{3}\# \mathbb{R} P^{3}$ as  quotient  space.  

 \end{itemize}
 
\end{theorem}

\subsection{Classification  of discrete  groups  of  isometries.} 
The previous example gives us the general behaviour for discrete groups of isometries on $S^2 \times \R$ as seen by the following Lemma 
\begin{lemma}\label{lema:discrete_subgroups_S2_R}
If  $\Gamma \subset Iso(S^2 \times \R)$ is a discrete subgroup, then there is a finite group $F \subset O(3)$ and $\lambda \in \R$ such that the exact sequence
	\[	1 \rightarrow O(3) \rightarrow Iso(S^2 \times \R) \rightarrow \R \rtimes \Z_2	\]
induces an exact sequence $1 \rightarrow F \rightarrow \Gamma \rightarrow  L$, where $L$ is either $\lambda \Z$ or $\lambda \Z \rtimes \Z_2$.
\end{lemma}

\begin{proof}
As the group $O(3)$ is compact, the projection of the discrete group $\Gamma$ onto $Iso(\R)$ is discrete, so it is of the form $\lambda \Z$ or $\lambda \Z \rtimes \Z_2$, for some $\lambda \in \R$. As $O(3)$ can be seen as a closed subgroup of $Iso(S^2 \times \R)$, then the intersection of $\Gamma$ with $O(3)$ is a finite group, which we denote by $F$. The result thus follows from the product structure of $Iso(S^2 \times \R)$. In fact, $\Gamma$ is generated by $F$, $\Z_2 \subset Iso(\R)$ and the twisted translation subgroup $\{ (\sigma^n , n \lambda) \in O(3) \times \R : n \in \mathbb{N} \}$, for some $\sigma \in O(3)$.
\end{proof}
{\bf Connected Components}
\begin{theorem}\label{Thm:general_isometry_gropu_S2_R}
If $\Gamma \subset Iso(S^2 \times \R)$ is a discrete subgroup, such that $(S^2 \times \R) / \Gamma$ is compact, then $Iso((S^2 \times \R) / \Gamma)$ is up to finite index, a closed subgroup of $SO(3) \times S^1$. In particular, the connected component of the identity of the isometry group of the quotient can only be one of the three possibilities:
	\[	SO(3) \times S^1, \quad S^1 \times S^1, \quad \textrm{or} \quad S^1.	\]
\end{theorem}

\begin{proof}
By Lemma \ref{lema:discrete_subgroups_S2_R}, the discrete group $\Gamma$ is generated by a finite group of $O(3)$ and a twisted translation as in Example \ref{example:discrete_subgroup_S2_R}. The isometry group is compact, so it has a finite number of connected components and by Proposition \ref{prop:centralizer-normalizer}, the connected component of the identity can be computed using the centralizer, which always contains the $\R$-factor, so the result follows by examining the possible connected, closed subgroups of $SO(3)$.
\end{proof}

\subsection{Example of a non discrete subgroup of  isometries}
We may observe that the projection onto the $S^2$ factor of a discrete group of isometries need not be discrete as the following example shows:

\begin{example}\label{example:discrete_subgroup_S2_R}
If $\sigma \in SO(3)$ is a rotation with irrational angle along a fixed axis, so that the orbit $\{ \sigma^n(p) : n \in \mathbb{N} \}$ is dense in a circle, orthogonal to the rotation axis, for almost every $p \in S^2$, then the group given by twisted translations 
	\[	\{ (\sigma^n , n) \in O(3) \times \R : n \in \mathbb{N} \}	\]
is a discrete subgroup of $Iso(S^2 \times \R)$ with non-discrete projection on $Iso(S^2)$.
\end{example}

\section{Sol Geometry}

\subsection{Riemannian Geometry of Three Dimensional Sol-manifolds}
Sol-geometry is given by the solvable Lie group of upper-triangular matrices 
	\[	S = \left\{	
			\left(\begin{array}{ccc} e^t & 0 & x \\
						0 & e^{-t} & y \\
						0 & 0 & 1\end{array}\right) : x,y,t \in \R	\right\},	\]
which decomposes as a semidirect product $S =   [S,S] \rtimes (S / [S,S]) \cong  \R^2 \rtimes \R $. In global coordinates, the vector fields 
	\[	X_1(x,y,t) = (e^t, 0,0), \quad	X_2(x,y,t) = (0, e^{-t},0), \quad	X_3(x,y,t) = (0,0,1),	\] 
define a basis of left-invariant vector fields. We choose the left invariant Riemannian metric in $S$ having this basis as orthonormal basis, so that in our global coordinates, the metric has the expresion
	\[	ds^2 = e^{-2t} dx^2 + e^{2t} dy^2 + dt^2.	\]
The isometry group of this metric is generated by left translations
	\[	L_g : S \rightarrow S, \qquad L_g(h) = gh,	\]
and the group of reflections $(x,y,z) \rightarrow (\pm x, \pm y,\pm z)$, isomorphic to the Dihedral group $D_4$. In particular, $Iso(S)$ has eight connected components, with the connected component of the identity isomorphic to $S$ \cite{Scott}.
\subsection{Existence of lattices}

Consider  the  following
\begin{remark}\label{remark:unimodular}

a Lie group admits a lattice subgroup if and only if it is unimodular \cite{Ra}, and so, not every solvable group admit lattice subgroups.
\end{remark}
\begin{example}
A solvable group which is closely  related  to $S$ considered here, is the group of orientation preserving affine transformations on $\R$, given by
	\[	\mathrm{Aff}^+(\R) \cong \left\{	\left(\begin{array}{cc} e^t & x \\ 0 & 1 \end{array}\right) : x,t \in \R	\right\}.	\]
We could try for example, to exponentiate the set
\[	\Lambda = exp\left(\left\{\left(\begin{array}{cc} n & m \\ 0 & 0 \end{array}\right) : n,m \in \mathbb{N} \right\}\right)
		= \left\{\left(\begin{array}{cc} e^n & (m/n)(e^n - 1) \\ 0 & 1 \end{array}\right) : n,m \in \mathbb{N} \right\},
\]
however, such discrete set is not a subgroup and the group which generates is not discrete. The problem is that the group $\mathrm{Aff}^+(\R)$ is not unimodular, and in fact its modular function has the expression
	\[	\Delta : \mathrm{Aff}^+(\R) \rightarrow \R_+,	\qquad	\Delta \left(\begin{array}{cc} e^t & x \\ 0 & 1 \end{array}\right) = e^t,	\]
which is non-trivial. 
\end{example}

The solvable group $S$ is unimodular, so that it admits a lattice subgroup and an explicit way to construct a lattice is as follows: Consider a matrix $A \in SL_2(\Z)$, such that $tr(A) > 2$ and the group
	\[	\Gamma_A = \left\{\left(\begin{array}{cc} A^n & Z \\ 0 & 1 \end{array}\right):
		n \in \Z, \ Z \in M_{2 \times 1}(\Z)	\right\} \cong   \Z^2  \rtimes_A \Z .	\]

\begin{lemma}\label{lema:carac_lattices_sol}
For every $A \in SL_2(\Z)$, with $\mathrm{tr}(A)> 2$, $\Gamma_A$ is conjugated in $SL_3(\R)$ to a lattice subgroup in $S$, moreover, every lattice subgroup of $S$ is conjugated to one of such groups.
\end{lemma}

\begin{proof}
Suppose first that $A \in SL_2(\Z)$, with $\mathrm{tr}(A)> 2$, then there is a matrix $B \in SL_3(\R)$ such that
	\[	BAB^{-1} = \left(\begin{array}{cc} e^\lambda & 0 \\ 0 & e^{- \lambda}	\end{array}\right),	\]
for some $\lambda \neq 0$. We may define  $A^t = B^{-1} \left(\begin{array}{cc} e^{t \lambda} & 0 \\ 0 & e^{- t \lambda}	\end{array}\right) B$, so that $\Gamma_A$ is a discrete subgroup of the group
\[	S_A= \left\{\left(\begin{array}{cc} A^t & Z \\ 0 & 1 \end{array}\right):
		t \in \R, \ Z \in M_{2 \times 1}(\R)	\right\} \cong \R \ltimes \R^2,	\]
such that
	\[	1 \rightarrow \R^2 / \Z^2 \rightarrow S_A/\Gamma_A \rightarrow \R / \Z \rightarrow 1,	\]
thus, $\Gamma_A$ is a lattice in $S_A$. Observe that we have an isomorphisms of Lie groups via the conjugation
	\[	S_A \rightarrow S, \qquad X \mapsto \left(\begin{array}{cc} B & 0 \\ 0 & 1 \end{array}\right) X \left(\begin{array}{cc} B^{-1} & 0 \\ 0 & 1 \end{array}\right),	\]
and thus, a lattice in $S$.

Before proceeding with the proof of this Lemma, we need to prove the following discrete projection Lemma:

\begin{lemma}\label{prop:solv-discrete-projection}
If $\Gamma \subset S$ is a discrete subgroup, then its projection $\overline{\Gamma} \subset S/[S,S] \cong \R$ is also discrete.
\end{lemma}

\begin{proof}
An element $\gamma = (x,y,t)$, with $t \neq 0$, acts discretely by translations on the line $\left\{	\left(\frac{x}{1-e^t}, \frac{y}{1-e^{-t}}, s \right): s \in \R	\right\} \subset S$, as $\gamma^n \left(\frac{x}{1-e^t}, \frac{y}{1-e^{-t}}, s \right) = \left(\frac{x}{1-e^t}, \frac{y}{1-e^{-t}}, s + n t \right)$. Thus, if $\Gamma$ is commutative, either $\Gamma \subset \R^2$ and its projection is trivial, or $\Gamma$ preserves a unique line on which it acts as translations and the action on this line is precisely the action on $\R$ of its projection, which must be discrete. If $\Gamma$ is non-commutative, then at least it has two elements $u= (a,b,0)$ and $\gamma = (x,y,t)$ with $t \neq 0$. Observe that if $b = 0$, then $\gamma u \gamma^{-1} = (e^t a -a, 0 ,0)$ and by iterating conjugation we get a non-discrete subgroup of $\R^2 \cap \Gamma$ which is impossible and the same goes for the case $a = 0$. Thus $a,b \neq 0$ and $\Gamma \cap \R^2$ contains two linearly independent vectors, say $u$ and $v = \gamma u \gamma^{-1} = (e^t a, e^{-t}b,0)$, which implies that $\Gamma \cap \R^2$ is cocompact in $\R^2$. $\Gamma / (\Gamma \cap \R^2)$ is discrete in $S/ (\Gamma \cap \R^2)$ and the projection $S /(\Gamma \cap \R^2) \rightarrow S / \R^2$ has compact Kernel $\R^2 / (\Gamma \cap \R^2) \cong S^1 \times S^1$, thus the corresponding projection of $\Gamma /(\Gamma \cap \R^2)$ into $S / [S,S]$ is discrete.
\end{proof}

Suppose now that $\Gamma \subset S$ is a lattice subgroup, then by Lemma \ref{prop:solv-discrete-projection}, the $\Gamma$ projects to a non-trivial discrete group $\R$, generated by an element $e^{n \beta}$, with $\beta \neq 0$. The intersection $\Gamma \cap \R^2$ is a lattice, so that there are $u, v \subset \R^2$ linearly independent, such that $\Gamma \cap \R^2 = \{n u + mv : n, m \in \Z\}$. Take $C \in GL_2(\R)$ the matrix sending $\Gamma \cap \R^2$ onto the canonical lattice $\Z^2$ and define the matrix $A' = B \left(\begin{array}{cc} e^{\beta} & 0 \\ 0 & e^{- \beta} \end{array}\right) B^{-1}$. An element $g = \left(\begin{array}{cc} B^{-1} A' B & W \\ 0 & 1 \end{array}\right) \in \Gamma$ must preserve $\Gamma \cap \R^2$, so that the element $h = \left(\begin{array}{cc} A' & BW \\ 0 & 1 \end{array}\right) \in \Gamma$ must preserve $\Z^2$. Observe that the action of $h$ in an element $v = (n,m) \in \Z^2$ is $A'v + BW \in \Z^2$, this implies that $BW \in \Z^2$ and $A' \in SL_2(\Z)$. In particular, the group $\Gamma$ is isomorphic to the group $\Gamma_{A'}$ and the isomorphism is obtained by conjugation.
\end{proof}

\begin{remark}\label{remark:solv-Q-structure}
The existence of lattices in the Lie group $S$ is related to the existence of a $\mathbb{Q}$-structure on $S$. More precisely, if $A \in SL_2(\Z)$, with $\mathrm{tr}(A) > 2$, and $c = \sqrt{\mathrm{tr}(A)^2 - 4}$, then $A$ is diagonalizable over the field $\mathbb{Q}(c)$, that is, there is a matrix $B \in SL_2(\mathbb{Q}(c))$ such that $BAB^{-1}$ is diagonal. If $Q_{ij}(X) = X_{ij}$ is the linear map that gives  the $(i,j)$-entry, then the group
	\[	\mathbb{G}(k) = \left\{\left(\begin{array}{cc} X & Z \\ 0 & 1 \end{array}\right):
		Z \in M_{2 \times 1}(k), \ Q_{ij}(BXB^{-1}) = 0, \ i \neq j, i,j \in \{1,2\} 	\right\},	\]
is algebraic subgroup of $SL_3(\R)$, defined by polynomial equations with coefficients over $\mathbb{Q}(c)$, such that $\mathbb{G}(\R) \cong S$ and $\mathbb{G}(\Z) = \Gamma_A$. Moreover, the Galois automorphism $\sigma : \mathbb{Q}(c) \rightarrow \mathbb{Q}(c)$, defined by $\sigma(c) = -c$, has a natural extension to automorphisms of matrices and polynomials, so that we have the embedding
	\[	SL_3(\mathbb{Q}(c)) \rightarrow SL_3(\R) \times SL_3(\R), \qquad Y \mapsto (Y, \sigma(Y)), 	\]
and a polynomial condition $Q(Y) = 0$ on $Y \in SL_3(\mathbb{Q}(c))$ is equivalent to the pair of polynomial conditions $Q(Y)+ \sigma(Q)(Y') = 0$ and $Q(Y) \sigma(Q)(Y') = 0$ on $(Y,Y') \in SL_3(\R) \times SL_3(\R)$, but the latter are polynomials with coefficients over $\mathbb{Q}$ (this trick is called ``restriction of scalars'' \cite{morris}).
\end{remark}

\begin{lemma}\label{lema:solv-centralizer}
$S$ has trivial center and the centralizer of a lattice group $\Gamma \subset S$ is also trivial.
\end{lemma}

\begin{proof}
Take $\gamma = (x,y,t)$ in the centralizer of $\Gamma$ in $S$, then as in the previous proposition $\Gamma \cap [S,S]$ has a rank two subgroup, thus it contains at least a vector $u = (a,b,0)$ such that $a,b \neq 0$ and we have
	\[	\gamma u \gamma^{-1} = (e^t a , e^{-t} b, 0 ) = (a,b,0),	\]
which implies that $t = 0$. As $\Gamma$ projects to a lattice group in $S / [S,S] \cong \R$, then there is a $\beta \in \Gamma$ such that $\beta = (c,d,s)$ with $s \neq 0$ and thus
	\[	\beta \gamma \beta^{-1} = (e^s x, e^{-s} y, 0) = \gamma  = (x,y,0)	\]
which implies that $x= y= 0$ and $\gamma$ is the identity. A completely analogous computation shows that $S$ has trivial center.
\end{proof}

\begin{corollary}\label{cor:solv-isometry-group}
If $\Gamma$ is a discrete group of isometries of $S$ such that $S/\Gamma$ has finite volume, then $S/\Gamma$ is compact and has finite isometry group. 
\end{corollary}

\begin{proof}
As the connected component of the isometry group of $S$ is $S$ itself acting by left multiplications, $\Gamma$ is modulo a finite index subgroup a lattice in $S$ and it lies in an exact sequence
	\[	1 \rightarrow \Gamma_0 \rightarrow \Gamma \rightarrow \Gamma_1 \rightarrow 1	\]
where $\Gamma_0 = \Gamma \cap [S,S]$ and $\Gamma / \Gamma_0 \cong \Gamma_1 \subset \R$. By Proposition \ref{prop:solv-discrete-projection} $\Gamma_1$ is a discrete subgroup, so this exact sequence induces a the fiber bundle
	\[	\R^2 / \Gamma_0 \rightarrow S / \Gamma \rightarrow R / \Gamma_0,	\]
so that $S/ \Gamma$ has finite volume if and only if $\R^2 / \Gamma_0$ and $\R / \Gamma_1$ are torus of the corresponding dimension and $S/\Gamma$ is compact. The isometry group of $S/\Gamma$ is a compact Lie group with connected component of the identity determined by the centralizer of $\Gamma$ in $S$ (Proposition \ref{prop:centralizer-normalizer}) which is the trivial group by Lemma \ref{lema:solv-centralizer}, thus the isometry group is a compact, zero-dimensional Lie group, i.e. finite.
\end{proof}

\subsection{Examples}
\begin{example}\label{ex:sol-lattice}
For $A = \left(\begin{array}{cc} 2 & 1 \\ 1 & 1 \end{array}\right)$ and $n \in \mathbb{N}$, consider the lattice $\Gamma_{A^n} =   \Z^2\rtimes_{A^n} \Z$. A matrix $Y = \left(\begin{array}{cc} M & W \\ 0 & 1 \end{array}\right) \in GL_3(\R)$ normalizes $\Gamma_{A^n}$ if and only if $M = A^k$ for some $k \in \Z$ and $(I - A^n) W \in \Z^2$, so that if $\Lambda_n = (I - A^n)^{-1} \Z^2$, then the normalizer is $N_{Iso(S)}(\Gamma_{A^n}) = \Z \ltimes_A \Lambda_n$ and the isometry group is computed as
	\[	Iso(S/\Gamma_{A^n}) =  (\Lambda_n/\Z^2) \rtimes_A  \Z_n.	\]
Three ilustrative cases are
	\begin{enumerate}
		\item $\Lambda_1 = \Z^2$, so that $Iso(S/\Gamma_A)$ is trivial;

		\item $\mathrm{det}(I- A^2) = - 5$, so that $\Z^2 \leq \Lambda_2 \leq \frac{1}{5} \Z^2$ and each contention is of index $5$, in particular we have that $Iso(S/\Gamma_{A^2}) =  \Z_{5}\rtimes\Z_2  $;

		\item $\Lambda_5 = \frac{1}{11} \Z^2$, so that $Iso(S/\Gamma_{A^5}) =  (\Z_{11} \times \Z_{11})\rtimes_A \Z_5$.
	\end{enumerate}
\end{example}

\subsection{Classification of  free actions}
\begin{remark}
The  previous family of examples exhibits isometric actions of each finite cyclic group on a three dimensional solvmanifold. Moreover, we should notice that such actions are necesarily not free, since there exists a very rigid classification of free actions of finite groups on three dimensional manifolds with Nil and Sol structure, based on $p$-rank estimates and P.A. Smith Theory, \cite{jolee}, \cite{kooohshin}.
\end{remark}

\section{Hyperbolic geometry}

\subsection{Normalizers of Fuchsian groups.}

Denote by $\H^n$ the n-dimensional hyperbolic space and recall that the isometry group $Iso(\H^n)$ is a non-compact semisimple Lie group that can be identified with the group $PO(n,1)$.
We  begin  by  recalling  the following  properties  of  normalizers  of  discrete  subgroups  of isometries. 

\begin{lemma}\label{lema:hyperbolic-isometry-group}
If $\Gamma \subset Iso(\H^n)$ is a discrete subgroup such that $\H^n /\Gamma$ has finite volume, then the normalizer group
	\[	\Lambda = \{ g \in Iso(\H^n) : g h g^{-1} = h, \ \forall \ h \in \Gamma \} \subset Iso(\H^n)	\]
is discrete and $Iso(\H^n/ \Gamma)$ is a finite group.
\end{lemma}

\begin{proof}
Passing to a finite cover doesn't alter the outcome, so we may suppose that $\Gamma, \Lambda \subset O(n,1)$. By Proposition \ref{prop:centralizer-normalizer}, the connected component of $\Lambda$ lies inside the centralizer of $\Gamma$ in $O(n,1)$. Let $g \in O(n,1)$ centralizing $\Gamma$, then the polynomial
	\[	P_t : M_{n+1}(\R) \rightarrow M_{n+1}(\R), \qquad P_t(X) = g X g^{-1} - X	\]
vanishes at $\Gamma$ but by Borel's density Theorem (see \cite{Fu}), $\Gamma$ is Zariski dense in $O(n,1)$ and thus $P_t(O(n,1)) = 0$ which tells us that $g$ lies in the center of $O(n,1)$, which is finite. This tells us that $\Lambda$ is a discrete group that contains the lattice $\Gamma$, so $\Lambda$ is also a lattice in $O(n,1)$. If $F_\Lambda, F_\Gamma \subset \H^n$ are fundamental domains of the groups $\Lambda$ and $\Gamma$ correspondingly, so we have that 
	\[	|Iso(\H^n/ \Gamma)| = |\Lambda / \Gamma| = Vol(F_\Gamma) /Vol(F_\Lambda) < \infty. 	\]
\end{proof}

\begin{remark}\label{remark:realization_hyperbolic_finite_isom}
The previous result is stated for hyperbolic manifolds in Corollary 3, Section 12.7 of \cite{Rat} and for hyperbolic orbifolds in \cite{ratcliffe99}, where the hypotheses are that the discrete group is non elementary, geometrically finite and without fixed $m$-planes, for $m < n-1$. In Lemma \ref{lema:hyperbolic-isometry-group} we presented an argument using Zariski-density of the lattice group in $Iso(\H^n)$, which implies for example the non-existence of fixed $m$-planes. As seen in \cite{Greenberg}, every finite group can be realized as the isometry group of a compact hyperbolic surface as in Lemma \ref{lema:hyperbolic-isometry-group}.
\end{remark}

\subsection{Rank of isometries   and Lie  groups  acting  on hyperbolic  surfaces.  }

\begin{lemma}\label{lema:circle-actions-surfaces}
If $\Sigma$ is a compact, orientable surface of genus $g \geq 2$, then there are no faithful actions of the compact group $S^1$ on $\Sigma$.
\end{lemma}

\begin{proof}
	Suppose there is a faithful action $S^1 \times \Sigma \rightarrow \Sigma$, then perhaps after an averaging process, we may suppose that the action is isometric with respect to a Riemannian metric $h$. The existence of isothermal coordinates \cite{Yamada} tells us that there exists a complex structure in $\Sigma$ such that in holomorphic coordinates $z = x + i y$, the vector fields $\partial_x$ and $\partial_y$ are $h$-orthogonal. As the $S^1$-action is $h$-isometric, it preserves angles and orientation in the isothermal coordinates and thus it is an action by holomorphic transformations. By the uniformization Theorem, the universal cover of $\Sigma$ is the hyperbolic semiplane $\H^2 \subset \C$ and the holomorphic automorphisms of $\Sigma$ lift to holomorphic automorphisms of $\H^2$ which also are isometric automorphisms with respect to the hyperbolic metric. As a consequence of this, we have that the $S^1$-action preserves a hyperbolic metric in $\Sigma$ which has finite volume, because $\Sigma$ is compact, but this contradicts Lemma \ref{lema:hyperbolic-isometry-group}.
\end{proof}

\begin{corollary}\label{cor:isometry-group-surfaces}
If $\Sigma$ is a compact, orientable surface of genus $g \geq 2$ and $h$ is a Riemannian metric in $\Sigma$, then the isometry group $Iso(\Sigma,h)$ is finite.
\end{corollary}

\begin{proof}
	As $\Sigma$ is compact, the isometry group $G = Iso(\Sigma,h)$ is a compact Lie group. If $\g$ denotes the Lie algebra of $G$, then for every $X \in \g$, the one parameter group $\{exp(tX)\}$ is a commutative group whose closure is a compact, commutative Lie group with connected component of the identity isomorphic to a product $S^1 \times \cdots \times S^1$. As a consequence of this and the fact that $G$ has only has finitely many connected components, if $G$ is infinite, then it has a closed subgroup isomorphic to $S^1$, but this is impossible as is shown in Lemma \ref{lema:circle-actions-surfaces}.
\end{proof}
\subsection{Non-Classification of  finite hyperbolic  groups  of  isometries. }
\begin{remark}
    It  is proved  in \cite{kojima} that  every  finite  group  can be  realized as the isometry group  of a closed hyperbolic  manifold  of  dimension  three. 
    
\end{remark}

%\todo{Creo que  es  mejor mover  esta  subsecci\'on  a una  secci\'on intermedia  dicinedo que  se  va  a  usar  en los  dos  casos  siguientes. }

\section{Finer classification of 2-dimensional hyperbolic isometries. 
}\label{section:finer} 

Recall that in dimension two, the  group $SL_2(\R)$  acts on $\H^2$ by isometries  in the  form of M\"obius transformations, so that we have a realization of the orientation preserving isometries as $Iso(\H^2) \cong PSL_2(\R)$. 
\subsection{Classification of  elements in $SL_{2}$ according  to their  fixed  point  sets on the visual  compactification of $\mathbb{H}^{2}$. }
We  recall  the classification  of  elements in $SL_{2}(\mathbb{R})$
An element $A \in SL_2(\R)$ has as a characteristic polynomial $p_A(x) = x^2 - tr(A) x + 1$, and discriminant $tr(A)^2 - 4$. Thus, there are three dynamically different possibilities for the isometry of $\H^2$ generated by $A$, characterized by the sign of $tr(A) - 2$:

\begin{itemize}
	\item $tr(A) - 2 > 0$, where the matrix is conjugated to a diagonal matrix over $\R$, and thus, the conjugated isometry is contained in the one parameter group of isometries generated by 
		\[	\left\{ \mathrm{exp}\left( \begin{array}{cc} t & 0 \\ 0 & -t \end{array} \right) = \left( \begin{array}{cc} e^t & 0 \\ 0 & e^{-t} \end{array} \right) : t \in \R \right\}.	\]
	One isometry of this type is called hyperbolic, and the one-parameter group generated by this matrix is characterized by the property of having two fixed points in the boundary $S^1 = \partial \H^2$ and preserves a foliation determined by the two points and guided by the geodesic that joints the two points (in the case of diagonal matrices, this is $\{ 0 , \infty \}$).

	\item $tr(A) - 2 = 0$, where the matrix is conjugated over $\R$ to an upper triangular matrix, and thus, the conjugated isometry is contained in the one parameter group of isometries generated by
		\[	\left\{ \mathrm{exp}\left( \begin{array}{cc} 0 & t \\ 0 & 0 \end{array} \right) = \left( \begin{array}{cc} 1 & t \\ 0 & 1 \end{array} \right) : t \in \R \right\}.	\]
	One isometry of this type is called parabolic, and the one-parameter group generated by this matrix is characterized by the property of having one fixed point in the boundary $\partial \H^2$ and preserving the foliation of horocycles tangent to the fixed point (in the upper triangular case, the horocycles that are tangent to $\infty$ are just horizontal lines).

	\item $tr(A) - 2 < 0$, where the matrix is conjugated over $\R$ to a rotation matrix, so that the conjugated isometry is contained in the one parameter group of isometries generated by 
		\[	\left\{ \mathrm{exp}\left( \begin{array}{cc} 0 & -t \\ t & 0 \end{array} \right) = \left( \begin{array}{cr} \cos(t) & -\sin(t) \\ \sin(t) & \cos(t) \end{array} \right): t \in \R \right\}.	\]
	One isometry of this type is called elliptic, and the one-parameter group generated by this matrix is characterized by the property of having one fixed point in the interior of $\H^2$ and preserving a foliation of circles.
	\end{itemize}

\begin{lemma}\label{lema:isometries-commutativity-fixedpoints}
If $\alpha, \beta \in PSL_2(\R)$ are two non-trivial elements, then
	\begin{enumerate}
	\item $\alpha$ and $\beta$ commute if and only if $Fix(\alpha) = Fix(\beta)$,

	\item $C(\alpha) = \{ \beta \in PSL_2(\R) : \alpha \beta = \beta \alpha \} = \{ exp(tX) : t \in \R \}$, for some $X \in \mathfrak{sl}_2(\R)$. In particular $C(\alpha)$ is isomorphic to either $\R$ or $S^1$.
	\end{enumerate}
\end{lemma}

\begin{proof}
Suppose $\alpha \beta = \beta \alpha$, then $\beta(Fix(\alpha)) = Fix(\alpha)$ and $\alpha(Fix(\beta)) = Fix(\beta)$. If $\alpha$ is parabolic or elliptic, then it has only one fixed point and thus $Fix(\alpha) = Fix(\beta)$ and the same applies for $\beta$ either parabolic or elliptic. In the case where both $\alpha$ and $\beta$ are hyperbolic, we observe that $\beta$ cannot interchange two distinct elements of the boundary $S^1$, thus the property $\beta(Fix(\alpha)) = Fix(\alpha)$ implies $Fix(\alpha) = Fix(\beta)$. On the other hand, if $\alpha$ and $\beta$ have the same set of fixed points, then they are elements of the same one-parameter group, this is obvious when the fixed points are in standard configuration, that is $\{0,\infty\}$, $\{\infty\}$ or $\{i\}$ according if the element is hyperbolic, parabolic or elliptic; and in general it can be seen via a conjugation of matrices by sending the fixed points to the standard configuration. In particular $\alpha \beta = \beta \alpha$, because a one-parameter group is commutative and the result follows.
\end{proof}
\subsection{Discrete  subgroups of Isometries of  $SL_{2}(\mathbb{R})$.}
\begin{corollary}\label{cor:finite-index-commutative-isometries}
If $\Gamma \subset PSL_2(\R)$ is a subgroup such that it has the identity element as an accumulation point (equivalently $\Gamma$ is not a discrete subgroup) and $\Lambda \subset \Gamma$ is a non-trivial, normal and discrete subgroup, then there exists $\Gamma_1 \subset \Gamma$ commutative subgroup of finite index.
\end{corollary}

\begin{proof}
\textbf{$\Lambda$ is cyclic.} As $\Lambda$ is normal, for every $\gamma \in \Gamma$, the conjugation induces an automorphism
	\[	\Lambda \rightarrow \Lambda, \qquad g \mapsto \gamma g \gamma^{-1},	\]
and as $\Lambda$ is discrete and $\Gamma$ has the identity element as an accumulation point, for every $F \subset \Lambda$ finite set, there exist $\gamma \in \Gamma$ close enough to the identity such that $\gamma \neq e$ and $\gamma g = g \gamma$, for every $g \in F$. By the Lemma \ref{lema:isometries-commutativity-fixedpoints}, the group generated by $F$ is a discrete subgroup of the one-parameter group $C(\gamma)$ and thus it is a cyclic group. For $F_1 \subset F_2 \subset \Lambda$ any two distinct finite subsets, there are elements $g_j \in \Lambda$ such that $\langle g_j \rangle = \langle F_j \rangle$ and $\langle F_1 \rangle \subset \langle F_2 \rangle$ which implies that $g_1 = g_2^k$ for some $k$ and in particular $0 < |g_2| < |g_1|$. Now $\Lambda$ must be cyclic because otherwise we would have a sequence $\{g_j\} \subset \Lambda$ obtained as the generators of subgroups generated by an increasing tower of finite subsets of $\Lambda$ that converge to the identity.

\textbf{Existence of $\Gamma_1$.} Take $\alpha \subset \Lambda$ a generator of the group and as $\gamma \alpha \gamma^{-1}$ is again a generator of $\Lambda$, for every $\gamma \in \Gamma$, then the subgroup
	\[	\Gamma_1 = \{ \gamma \in \Gamma : \gamma \alpha \gamma^{-1} = \alpha \}	\]
is a finite index subgroup of $\Gamma$ ($[\Gamma : \Gamma_1] \leq 2$ if $\Lambda \cong \Z$, and $[\Gamma : \Gamma_1] \leq |\Lambda|$ if $\Lambda \cong \Z / m \Z$). Finally, by the Lemma \ref{lema:isometries-commutativity-fixedpoints}, $\Gamma_1$ is commutative and the result follows.
\end{proof}
\subsection{Non-classification of  finite  groups  of  isometries}
\begin{remark}
It  is  proved  in \cite{Greenberg} that  every finite  group  can  be  realized  as  the isometry  group  of a compact hyperbolic  surface. 
\end{remark}

\section{$\H^2 \times \R$}
Recall  \cite{kobayashinomizu}, Chapter  VI, Theorem  3.5 that  given  a   product  of riemannian   manifolds $M\times N$  with  $M$ of  constant  sectional  curvature  and  $N$ flat,  the   isometry  group  of $M\times  N$   decomposes  as  a  direct  product, $ {\rm Iso }(M)\times {\rm Iso} (N).$ The following  result gives us the isometry groups of finite volume quotients of $\H^2 \times \R$ (see Theorem \ref{teo:hyperbolic-fibration-isometry-group} for another proof):

\subsection{Isometry  groups  of  finite  volume. }
\begin{theorem}\label{teo:isometry-group-product}
If $G \subset Iso(\H^2 \times \R)$ is a discrete subgroup such that $(\H^2 \times \R) / G$ has finite volume, then the group $Iso((\H^2 \times \R) / G)$ is a finite extension of $S^1$
\end{theorem}

\begin{proof}
Consider the exact sequence
	\[	1 \rightarrow K \rightarrow G \rightarrow \Gamma \rightarrow 1,	\]
	where $K = G \cap Iso(\R)$ is a discrete subgroup of $G$ and $\Gamma \cong G / K$ is a subgroup of isometries of $\H^2$. If $\Gamma$ is discrete as a subgroup of $Iso(\H^2)$, then $\H^2 / \Gamma$ is an hyperbolic orbifold such that 
	\[	\R / K \rightarrow (\H^2 \times \R) / G \rightarrow \H^2 / \Gamma \]
is a locally trivial fiber bundle and as $(\H^2 \times \R) / G$ has finite volume, then $\R/K \cong S^1$ and $\Gamma$ is a Lattice subgroup of $Iso(\H^2)$. In this case, we have an exact sequence of isometry groups
	\[	1 \rightarrow Iso(S^1) \rightarrow Iso(\H^2 \times \R/G) \rightarrow Iso(\H^2 / \Gamma) \rightarrow 1,	\]
where $Iso(\H^2 / \Gamma)$ is a finite group by Lemma \ref{lema:hyperbolic-isometry-group} and thus $Iso(\H^2 \times \R/G)$ is a finite extension of $S^1$.

If $\Gamma$ is not discrete as a subgroup of $Iso(\H^2)$, we can see that the quotient $(\H^2 \times \R) / G$ cannot have finite volume. To see this, first observe that we have another exact sequence
	\[	1 \rightarrow \Lambda \rightarrow G \rightarrow L \rightarrow 1,	\]
where $\Lambda = G \cap Iso(\H^2) \subset \Gamma$ is a discrete, normal subgroup and $G/\Lambda \cong L \subset Iso(\R)$. If $\Lambda = 0$, then $G \cong L$ is commutative and thus $\Gamma$ is commutative. If instead $\Lambda$ is non-trivial, then Corollary \ref{cor:finite-index-commutative-isometries} tells us again that $\Gamma$ is commutative (perhaps after passing to a finite index subgroup). In any case, $G$ leaves a closed surface $\zeta \times \R \subset \H^2 \times \R$ fixed, where $\zeta$ is a geodesic, an horocycle or a circle (corresponding to the type of the isometries of $\Gamma$). If $\Gamma$ consists of parabolic or hyperbolic elements, then $\Gamma$ acts discretely by Euclidean automorphisms in  $\zeta \times \R \cong \R^2$ so that by Bieberbach Theorem \cite{Rat}, $\Gamma$ contains a finite index subgroup isomorphic to a subgroup of $\Z^2$ and in particular the fundamental domain of the $G$-action in $\H^2 \times \R$ contains a subset isometric to
	\[	\{ x + iy : a < x < b \} \times [c,d] \subset \H^2 \times \R,	\]
this implies that $(\H^2 \times \R) /G$ doesn't have finite volume. If $\Gamma$ consists of elliptic elements, then $G$ acts discretely by Euclidean automorphisms in $\zeta \times \R \cong S^1 \times \R$, and thus as in the previous case, the $G$-action has a fundamental domain containing an open subset isomorphic to
	\[	\{ (s e^{i \theta}, r) \in \mathbb{D} \times \R : a < \theta < b, \ c < r < d \},	\]
where $\mathbb{D} \cong \H^2$ is the Poincar\'e disc model of the hyperbolic plane, and again $(\H^2 \times \R) /G$ doesn't have finite volume.
\end{proof}

\section{$\widetilde{SL_2}$ Geometry}\label{section:SL}

\subsection{Riemannian Geometry of $PSL_2(\R)$}
Riemannian structure of
Recall that given a Riemannian manifolds $(M,g)$, there is a natural construction of a Riemannian metric tensor on the tangent bundle $TM$ constructed as follows: if $(p,x) \in T M$, and $(c(t),v(t)) \in T M$ is a smooth curve such that $c(0) = p$ and $v(0) = x$, then
	\[	\| (c'(0), v'(0)) \|_{(p,x)}^2 = 	\| d \pi_{(p,x)} ((c'(0), v'(0))) \|_p^2 + \left\|\frac{D}{dt}_{|t=0} v(t) \right\|^2_p,	\]
where $\pi : TM \rightarrow M$ is the projection, $\frac{D}{dt} v(t)$ is the covariant derivative along the curve $c(t)$ and $g(u,u)_p = \|u\|_p^2$. If $X = c'(0)$ and $Z = v'(0)$, in local coordinates we have the formula
	\[	\| (X,Z) \|_{(p,x)}^2 = 	\| X \|_p^2 + \left\| Z + X^jv^i \Gamma_{ij}^k \partial_k \right\|^2_p.	\]
The vector $(X,Z)$ is called horizontal if $c(t)$ is constant, and thus $X=0$, it is called vertical if it is orthogonal to every horizontal vector in which case $Z = - X^jv^i \Gamma_{ij}^k \partial_k$. So, we have a decomposition in horizontal and vertical components as
	\[	(X,Z) = (0, Z + X^jv^i \Gamma_{ij}^k \partial_k ) + (X, -X^jv^i \Gamma_{ij}^k \partial_k ).	\]
If we take the global coordinates $(x,y) \mapsto x+ iy$ of the hyperbolic plane
	\[	\H^2 = \{ z \in \C : Im(z) > 0 \},	\]
with corresponding metric tensor $ds^2 = \frac{dx^2 + dy^2}{y^2}$, then the Christoffel symbols at a point $x + i y$ are given by
	\[	-\Gamma_{11}^2 = \Gamma_{22}^2 = \Gamma_{12}^1 = \Gamma_{21}^1 = - 1/y.	\]
There is a natural identification of the tangent bundle
	\[	\H^2 \times \C \cong \textrm{T} \H^2, \qquad (z,w) \mapsto \frac{d}{dt}_{|t=0}(z + tw)	\]
and so the projection $\pi : T \H^2 \rightarrow \H^2$ is just given by the projection in the first factor and we have global coordinates in each tangent plane $\partial_1 = 1$ and $\partial_2 = i$. If as before, $(X,Z)$ is a tangent vector to $T \H^2$ at the point $(p,v) = (i,1)$, then the orthogonal decomposition in horizontal and vertical components is given by
	\[	(X,Z) = (0, Z - X^2 + X^1 i) + (X, X^2 - X^1 i).	\]
The isometric action by M\"obius transformations of $SL_2(\R)$ in $\H^2$, induces the action in the tangent bundle 
	\[	SL_2(\R) \times T \H^2 \rightarrow T \H^2, \quad 
		\left(\begin{array}{cc} a & b \\ c & d \end{array}\right) \cdot (z,w) = \left( \frac{az + b}{cz + d}, \frac{w}{(cz + d)^2} \right).	\]
This action is transitive in the unitary tangent bundle $T^1 \H^2 = \{(z,w) \in \H^2 : \|w\|_z = 1 \}$, so the orbit of the point $(i,1) \in T^1 \H^2$ induces the diffeomorphism $\phi : PSL_2(\R) \rightarrow T^1 \H^2$ given explicitly by the formula
	\[	\phi \left(\begin{array}{cc} a & b \\ c & d \end{array}\right) = \left( \frac{a i + b}{c i + d}, \frac{1}{(c i + d)^2} \right). \]
As this action is also isometric with respect to the previously defined metric, it will define a left invariant metric in $PSL_2(\R)$ that corresponds to an inner product in its tangent vector to the identity, naturally identified with the Lie algebra
	\[	\mathfrak{sl}_2(\R) = \{ A \in M_2(\R) : \textrm{tr}(A) = 0 \}.	\]
More precisely, if we consider the derivative $d \phi$, we get the identification 
	\[	\Psi : \mathfrak{sl}_2(\R) \rightarrow T_{(i,1)} (T \H^2), \qquad \Psi(X) = \frac{d}{dt}_{|t=0} \phi(exp(tX)).	\]
A basis of $\mathfrak{sl}_2(\R)$ is given by
	\[	X_1 = \left(\begin{array}{cc} 1 & 0 \\ 0 & -1 \end{array}\right), \quad
		X_2 = \left(\begin{array}{cc} 0 & 1 \\ -1 & 0 \end{array}\right), \quad X_3 = \left(\begin{array}{cc} 0 & 1 \\ 1 & 0 \end{array}\right). \]
If $g_{t,j} = exp(t X_j)$, then $\phi(g_{t,1}) = (e^{2t} i, e^{2t})$, $\phi(g_{t,2}) = (i , e^{2 i t})$ and
	\[	\phi(g_{t,3}) = \left( \frac{ch(t) i + sh(t)}{ch(t) + i sh(t)}, \frac{1}{(ch(t) + i sh(t))^2} \right), \]
where $ch(t)$ and $sh(t)$ denote the hyperbolic cosine and the hyperbolic sine correspondingly. If $\widehat{X}_j = \Psi(X_j)$, we have
	\[	\widehat{X}_1 = (2i,2), \quad \widehat{X}_2 = (0,2i), \quad \widehat{X}_3 = (2,-2i), 	\]
where we immediatly see that $\widehat{X}_2$ is vertical and a direct computation tells us that $\widehat{X}_1$ and $\widehat{X}_3$ are horizontal and orthogonal. Thus $\{\frac{1}{2} X_1, \frac{1}{2} X_2, \frac{1}{2} X_3 \}$ is an orthonormal basis in the corresonding inner product in $\mathfrak{sl}_2(\R)$.

As the $PSL_2(\R)$-action is given by holomorphic maps, it commutes with the action of $S^1$ given by rotations in each tangent plane
	\[	S^1 \times T^1 \H^2 \rightarrow	T^1 \H^2, \qquad \eta \cdot (z,w) = (z,\eta w),	\]
as well as with the map $(z,w) \mapsto (\overline{z},\overline{w})$. It is immediate that the previous maps act by isometries and in fact generate the whole isometry group. Thus, the isometry group $Iso(PSL_2(\R))$ is isomorphic to $PSL_2(\R) \times (S^1 \rtimes \Z_2)$, see \cite{Scott}.
\subsection{Groups  of  isometries  of  finite  volume}
\begin{theorem}\label{teo:psl2-orbifolds-isometries}
If $\Gamma \subset Iso(PSL_2(\R))$ is a discrete group such that $PSL_2(\R) / \Gamma$ has finite volume, then
	\[ Iso(PSL_2(\R) / \Gamma) \cong S^1 \rtimes F,	\]
where $F$ is a finite group.
\end{theorem}

\begin{proof}
Consider the projection into the simple factor
	\[	P : Iso(PSL_2(\R)) \rightarrow PSL_2(\R),	\]
as the Kernel of $P$ is compact and $\Gamma$ is a discrete subgroup, then $\Gamma_0 = P(\Gamma)$ is a discrete subgroup of $PSL_2(\R)$ and $\Gamma_0 \cong \Gamma / F_0$, with $F_0 = Ker(P) \cap \Gamma$ a finite subgroup of $S^1 \rtimes \Z_2$. Observe that $\pi : PSL_2(\R)	\rightarrow \H^2$ is a fiber bundle with fiber $S^1$ such that $\pi(\gamma x) = P(\gamma) \pi(x)$, so we have an induced projection
	\[	\pi : PSL_2(\R) / \Gamma \rightarrow \H^2 / \Gamma_0 ,	\]
which implies that $\H^2 / \Gamma_0 $ has finite hyperbolic area. As we also have the identification $\pi:  PSL_2(\R) / \Gamma_0  \rightarrow \H^2/\Gamma_0$, we have that $\Gamma_0$ is a Lattice in $PSL_2(\R)$. By Lemma \ref{lema:hyperbolic-isometry-group}, we have that $\Gamma_0$ has finite index in $\Lambda = N_{PSL_2(\R)}(\Gamma_0)$. Observe that if $F_1 = \Lambda / \Gamma_0$, then we have that $\Gamma \subset \Lambda \times S^1 \rtimes \Z_2$ and a bijection of sets
	\[	(\Lambda \times S^1 \rtimes \Z_2) / \Gamma \cong \frac{(\Lambda \times S^1 \rtimes \Z_2) / F_0 }{\Gamma / F_0 } \hookrightarrow (\Lambda \times S^1 \rtimes \Z_2) / \Gamma_0  \cong F \ltimes S^1,  \]
where $F$ is either $F_1$, or $F_1 \times \Z_2$, deppending on whether $\Gamma$ contains the map $(z,w) \mapsto (\overline{z},\overline{w})$ or not. Thus, we have that
	\[	S^1 \subset N_{Iso(PSL_2(\R))}(\Gamma) / \Gamma \subset (\Lambda \times S^1 \rtimes \Z_2)/\Gamma \cong F \ltimes S^1,	\]
and the result follows.
\end{proof}
\subsection{Non-classification of  finite group  actions. }
\begin{remark}
As seen in Remark \ref{remark:realization_hyperbolic_finite_isom}, we can obtain every finite group as an isometry group of an hyperbolic surface, so that, the finite factor of the isometry group in Theorem \ref{teo:psl2-orbifolds-isometries}, can be any finite group.
\end{remark}

\subsection{Isometries of the universal cover $\widetilde{SL_2}(\R)$.}

The Lie group $PSL_2(R)$ is topologically the product $S^1 \times \IR^2$, so that there is a simply connected Lie group denoted by $\widetilde{SL_2}(\R)$ which is the topological universal cover of $PSL_2(\R)$ and algebraically it is a  non-split central extension by a cyclic group $\Z$, more precisely, there is an exact sequence
	\[	1 \rightarrow \Z \rightarrow \widetilde{SL_2}(\R) \rightarrow PSL_2(\R) \rightarrow 1,	\]
where $\Z \subset \widetilde{SL_2}(\R)$ lies in the center. We can pull-back the metric tensor of $PSL_2(\R)$, constructed in the previous section, to $\widetilde{SL_2}(\R)$ to obtain the model of the homogeneous 3-dimensional geometry denoted by $SL_2$.

\begin{remark}
The isometry group of $\widetilde{SL_2}(\R)$ can be characterized in three different ways. First, we have the homomorphism
	\[	\widetilde{SL}_2(\R) \times \R \rightarrow Iso(\widetilde{SL}_2(\R))	\]
given by left and right multiplications, here $\R \cong \widetilde{SO}(2)$ is the universal cover of the rotation group $SO(2) \subset SL_2(\R)$, with Kernel $\Z = \R \cap \widetilde{SL}_2(\R)$ being precisely the center of $\widetilde{SL_2}(\R)$. The group $Iso(\widetilde{SL}_2(\R))$ has two connected components and 
	\[	Iso(\widetilde{SL}_2(\R))_0 \cong \big( \widetilde{SL}_2(\R) \times \R \big) / \Z	\]
is the component of the identity. In fact, we have an epimorphism
	\[	Iso(\widetilde{SL}_2(\R)) \rightarrow Iso(PSL_2(\R)) \cong PSL_2(\R) \times (S^1 \rtimes \Z_2),	\]
with kernel isomorphic to $\Z$, however, the group $Iso(\widetilde{SL}_2(\R))$ is no longer a product group. The left projection of the previous product gives us the second description in terms of a short exact sequence
	\[	1 \rightarrow \R \rightarrow Iso(\widetilde{SL}_2(\R))_0 \rightarrow PSL_2(\R) \rightarrow 1,	\]
and if we consider the groups $\widetilde{SL}_2(\R)$ and $\R$ as closed subgroups of $Iso(\widetilde{SL}_2(\R)$, then we have the third description
	\[	Iso(\widetilde{SL}_2(\R))_0 = L(\widetilde{SL}_2(\R)) R(\R),	\]
where $L(\cdot)$ and $R(\cdot)$ represent left and right multiplications in the group $\widetilde{SL}_2(\R)$.
\end{remark}

A discrete subgroup $\Gamma \subset Iso(PSL_2(\R))$ can be lifted to a discrete subgroup $\widetilde{\Gamma} \subset Iso(\widetilde{SL}_2(\R))$, so that $\widetilde{SL}_2(\R) / \widetilde{\Gamma} \cong PSL_2(\R) / \Gamma$ and thus, we can compute $Iso(\widetilde{SL}_2(\R) / \widetilde{\Gamma})$ with Theorem \ref{teo:psl2-orbifolds-isometries}, however, not every discrete group of $Iso(\widetilde{SL}_2(\R))$ can be obtained this way. In the next section we discuss the proof in the general setting for discrete groups of isometries in $\widetilde{SL}_2(\R)$.

The following Lemma is well known and holds for every Lie group, but we include a proof of the case we need for the sake of completeness.

\begin{lemma}\label{lema:stable-neighborhoods-commutator-1}
If $G$ is a Lie group locally isomorphic to $\R \times SL_2(\R)$, for example $G$ can be the isometry group of $\widetilde{SL}_2(\R)$ or $\H^2 \times \R$, then there exists a neighborhood of the identity $U \subset G$ such that $[U, U] \subset U$.
\end{lemma}

\begin{proof}
Observe first that this is a local property, so we only need to prove this for linear groups. As the $\R$ factor lies in the center, we have that
	\[	[g g_0, h h_0 ] = [g,h], \qquad \forall \ g_0, h_0 \in \R \]
and thus we only need to prove this for $SL_2(\R)$. The commutator
	\[	\left[ \left(\begin{array}{cc} a & x \\ y & b \end{array}\right) , \left(\begin{array}{cc} c & z \\ w & d \end{array}\right) \right] = \left(\begin{array}{cc} t_1 & t_3 \\ t_4 & t_2 \end{array}\right),	\]
is defined by the relations 
	\begin{itemize}
		\item $t_1 = 1 + xy + zw + xyzw + wxac + w^2 x^2 - adxw - a^2zw + bxwd - yd^2x - azyd$,
	
		\item $t_2 = 1 + xy + zw + xyzw -xwbc - c^2xy  + ayzc - zb^2w - zbcy + zybd + z^2y^2$,

		\item $t_3 = xac(d-c) - cx^2w + acz(a-b) - xwbz + zydx + z^2ya$,

		\item $t_4 = w^2xb + wxcy + bdw(b-a) - awyz + bdy(c-d) - dy^2 z$.
	\end{itemize}
So that if $0 \leq |x|,|y|,|z|,|w| < \varepsilon$ and $1 - \varepsilon < a,b,c,d < 1 + \varepsilon$, then there is a constant $C>0$ independent of $\varepsilon$ such that $|t_3|, |t_4| < C \varepsilon^2$ and $|t_1-1|, |t_2 - 1| < C \varepsilon^2$. Thus, by choosing $\varepsilon > 0$ such that $C \varepsilon^2 < \varepsilon$, the neighborhood
	\[	U_\varepsilon = \left\{ \left(\begin{array}{cc} a & x \\ y & b \end{array}\right) : |x|,|y| < \varepsilon, \ |a-1|,|b-1| < \varepsilon \right\}	\]
is stable under taking commutators.
\end{proof}

\subsection{Isometry  groups  of finite  volume}

\begin{proposition}\label{prop:hyperbolic-discrete-projections}
Let $H$ be a Lie group which is a central extension of $PSL_2(\R)$ of the form
		\[	1 \rightarrow \R \rightarrow H \rightarrow PSL_2(\R) \rightarrow 1.	\]
If $G \subset H$ is a discrete subgroup with induced exact sequence
	\[	1 \rightarrow K \rightarrow G \rightarrow \Gamma \rightarrow 1,	\]
with $K \subset \R$, then either $\Gamma \subset PSL_2(\R)$ is discrete or is an abelian subgroup leaving fixed a point, a geodesic or a horocycle in $\H^2$.
\end{proposition}

\begin{proof}
Denote by $p : H \rightarrow PSL_2(\R)$ the projection and consider $U \subset H$ a neighborhood of the identity such that $[U,U] \subset U$ and $U \cap G = \{e\}$. We have that the group $L = \langle p(U) \cap \Gamma \rangle$ is a commutative subgroup of $PSL_2(\R)$, to see why this is true take two elements $\alpha, \beta \in G$ such that $p(\alpha), p(\beta) \in p(U)$, then we may write those elements as $\alpha = \alpha_0 \alpha_1$ and $\beta = \beta_0 \beta_1$, where
	\[	\alpha_1,\beta_1 \in \R, \qquad \alpha_0, \beta_0 \in U.	\]
As $\R$ lies in the center of $H$ we have that $[\alpha_0,\beta_0] = [\alpha,\beta] \in G \cap U$ and thus $\alpha$ and $\beta$ commute. Now, for every $\alpha \in G$, choose a neighborhood of the identity $U_\alpha \subset G$ such that $[\alpha, U_\alpha] \subset U$, so that the elements of $\Gamma \cap p(U_\alpha)$ commute with $p(\alpha)$ (same argument as with the commutativity of $L$). Suppose that $\Gamma$ is non-discrete, then $L$ is a non-trivial commutative subgroup and for every $\gamma = p(\alpha) \in \Gamma$, we have that $\Gamma \cap p(U_\alpha)$ is a non-trivial subset that generates the group $L$ and commutes with $\gamma$. So, $\Gamma$ commutes with $L$ and thus, there exists an element $X \in \mathfrak{sl}_2(\R)$ such that 
	\[	\Gamma \subset \overline{L} = \{exp(tX) : t \in \R \}	\]
and $\Gamma$ leaves fixed a point, a geodesic or a horocycle, depending on the type of $X$.
\end{proof}

\begin{theorem}\label{teo:hyperbolic-fibration-isometry-group}
If $G$ is a discrete subgroup of isometries of $X$ either $\widetilde{SL}_2(\R)$ or $\H^2 \times \R$ such that $X/G$ has finite volume, then the isometry group $Iso(X/G)$ is a finite extension of $S^1$.
\end{theorem}

\begin{proof}
The exact sequence
	\[	1 \rightarrow \R \rightarrow Iso(X) \rightarrow PSL_2(\R) \rightarrow 1	\]
induces the sequence
	\[	1 \rightarrow K \rightarrow G \rightarrow \Gamma \rightarrow 1,	\]
if $\Gamma$ is non-discrete, then it preserves a geodesic, a point or a horocycle by Proposition \ref{prop:hyperbolic-discrete-projections} and we can see that $X/G$ doesn't have finite volume (as we did in Theorem \ref{teo:isometry-group-product}). So, $\Gamma$ is a discrete subgroup of isometries of the hyperbolic plane and we have a fiber bundle structure
	\[	\R \rightarrow X \rightarrow \H^2	\]
so that the volume form decomposes as
	\[	\int_\R \int_{\H^2} f d \mu dt = \int_X f d vol_X	\]
where $d \mu$ is the hyperbolic area form. If $D \subset \H^2$ is a fundamental domain of $\Gamma$, then $\pi^{-1}(D) = \widehat{D}$ is such that $g \widehat{D} \cap \widehat{D} \neq \emptyset$ only for $g \in K = G \cap \R$. Thus for $\Omega \subset \R$ fundamental domain of $K$ in $\R$ we have that $\Omega \times D$ is a fundamental domain for $G$ which implies that
	\[	Vol(X/G) \geq \int_\Omega \int_D \xi = |\Omega| \times \mu(D),	\]
and we have that $\mu(A) < \infty$ and $|\Omega| < \infty$ which implies that $K = \Z$. Take $\widetilde{N} = N_{Iso(X)}(G)$ and $N = N_{PSL_2(\R)}(\Gamma)$, so that we have the exact sequence
	\[	1 \rightarrow N_0 \rightarrow \widetilde{N} \rightarrow N \rightarrow 1	\]
(because $gGg^{-1} = G$ projects to $\overline{g} \Gamma \overline{g}^{-1} = \Gamma$ and $N_0 = \R$ because $\R$ normalizes $\widetilde{SL_2}(\R)$), this sequence induces the exact sequence
	\[	1 \rightarrow N_0 / \Z \rightarrow \widetilde{N}/G \rightarrow N / \Gamma \rightarrow 1	\]
	(to see that this sequence is exact observe that $\pi(gG) = \overline{g} \Gamma$ is well defined and surjective, the condition $\pi (gG) = \Gamma$ holds if and only if $\overline{g} \in \Gamma$ and thus $g = [A,r]$ with $A \in \Gamma$, this is because there is an element $h = [A,s] \in G$, thus $gG = gh^{-1} G$ and $gh^{-1} \in i(N_0/\Z)$. This implies that the kernel of $\pi$ is $i(N_0 /\Z)$. Finally $i(r\Z) = G$ if and only if $r \in G$, but $G \cap \R = \Z$, so that $r \Z = \Z$ and thus $i$ is injective). This exact sequence can be written as
	\[	1 \rightarrow S^1 \rightarrow Iso(X/G) \rightarrow Iso(\H^2/\Gamma) \rightarrow 1	\]
which implies the result because of the Lemma \ref{lema:hyperbolic-isometry-group}.
\end{proof}

\section{Corollaries of  the Main  Theorem }\label{section:corollaries}

\subsection{Actions  of $SL_{k}(\mathbb{Z})$ on  aspherical three dimensional  manifolds  by  isometries.}
We  have  the  following  affirmative   solution to  Problem \ref{problem:shengkui}. 

\begin{theorem}

Any  group  action by  isometries   of $SL_{k+1}(\mathbb{Z})$, with  $k\geq 3$, on  a  closed, aspherical $3$-manifold factors  trough a  finite  group.  

\end{theorem}

\subsection{Discrete  groups  acting  with a  sufficiently  collapsed Alexandrov space as  quotient}

\begin{theorem}\label{cor:suffcollapsed}

Assume  that  a discrete group $\Gamma$ acts  by  isometries on the  three dimensional  Alexandrov  space $X$ such  that the  quotient $X/\Gamma$ is  sufficiently  collapsed  with parameters $d$, and $\epsilon$. Then, Theorem \ref{zimmer:3man},  together  with the  geometrization of  $3$-dimensional  Alexandrov  spaces provide  a  classification of the  possible  such $\Gamma$ within  the lattices in the  isometry  groups.  
\end{theorem}

\subsection{Hilbert Smith Conjecture  for  three dimensional Alexandrov  spaces }

Let  us  recall Theorem \ref{Theorem:pardon}. 
\begin{theorem}[\cite{pardon1}, \cite{pardon}]
For every prime $p$, there are no faithful actions by homeomorphisms of the $p$-adic group $\widehat{\Z}_p$ on a topological manifold of dimension $n \leq 3$.
\end{theorem}

As seen in section\ref{section:alex}, an Alexandrov space $X$ has a closed subset $S_X$, corresponding to topologically singular points and such that the set of regular points $R_X = X \setminus S_X$ is an open-dense subset, having the structure of a topological manifold. An action by homeomorphisms on $X$ must preserve the decomposition $X = S_X \cup R_X$ and a continuous action of $\widehat{\Z}_p$ which is trivial on the regular points, is trivial on the whole space $X$.  

Hence, the weaker version of the $p$-adic Hilbert-Smith conjecture for Alexandrov spaces holds. 

A  consequence of Theorem \ref{Theorem:pardon}, gives us

\begin{theorem}\label{Theorem:HS-for-Alex}
If $G$ is a locally compact, topological group, acting faithfully on a three dimensional Alexandrov space by homeomorphisms, then $G$ is a Lie group.
\end{theorem}

\begin{remark}\label{remark:dense_singular}
As observed in previous section, there is a subset of metrically regular points which admits a compatible Riemannian metric, constructed in \cite{OtsuShioya}. Thus, we have as a consequence of Theorem \ref{Theorem:RepScep}, that the $p$-adic group $\widehat{\Z}_p$ cannot act faithfully by bi-Lipschitz homeomorphisms. However, we should be careful, as the set of metrically singular points can be dense, as seen in an example constructed in \cite{OtsuShioya} as a limit of Alexandrov spaces, using baricentric subdivisions of a tetrahedron.
\end{remark}

\subsection{Non Existence  of actions of   Higher  Rank Lattices  by Isometries on three dimensional Geometric Orbifolds }

\begin{theorem}
Let $\Gamma$ be a higher rank lattice acting by isometries on a finite volume, three dimensional orbifold $X$ (modelled over a homogeneous 3-manifold $X$), then the action factors through a finite group if either:
\begin{itemize}
    \item $X$ is aspherical or,

    \item $\Gamma$ is non-uniform.
\end{itemize}
As an example of this, we have $\Gamma = SL_{r}(\mathbb{Z})$ with  $r\geq 3$.
\end{theorem}

\subsection{Characterization of  Higher Rank Lattices actions  by  isometries  on three dimensional  spherical Orbifolds }

\begin{remark}\label{Existence:Actions_round_sphere}
As a consequence of the previous discussion, for every semisimple Lie group $G$ such that $G \times SO(4)$ is isotypic\footnote{For example, any product $G = G_1 \times \cdots \times G_k$, where each $G_j$ is one of $SO(3,1)$, $SO(2,2)$ or $SO(4,\C)$.}, there is an irreducible lattice $\Gamma \subset G$ and an homomorphism $\Gamma \rightarrow SO(4)$ with dense image. In particular, such lattice acts by isometries on the round sphere $S^3$ with dense orbits. This tells us that there is  no restriction on the dimension of the class of  higher rank lattices which can act on the round sphere, but the type of such lattice is restricted. The same applies to the $3$-orbifolds of the type $S^2 \times S^1$, as the first factor has isometry group $SO(3)$ which is simple.
\end{remark}

We have in fact a converse of Remark \ref{Existence:Actions_round_sphere}, given by the following Theorem:

\begin{theorem}\label{teo:infinite_image_rep}
If $\Gamma \subset G$ is a lattice in a higher rank, semisimple Lie group, $K$ is a compact Lie group and $\varphi : \Gamma \rightarrow K$ is a homomorphism with dense image, then the group $G \times K$ is isotypic and $\Gamma \subset G$ is cocompact.
\end{theorem}

Theorem \ref{teo:infinite_image_rep} is ``well known to the experts'', but a sketch of the first part of its proof is made in \cite{brownfisherhurtado2}, section 2.3. The fact that the lattice $\Gamma$ is cocompact is a consequence of Godement's compactness criterion.

\begin{corollary}
If $\Gamma \subset G$ is a lattice in a higher rank, simple Lie group, $K$ is a compact Lie group and $\varphi : \Gamma \rightarrow K$ is a homomorphism with infinite image, then $G \times L$ is isotypic, with $L = \overline{\varphi(\Gamma)}$. In particular, $\mathrm{dim}(G) \leq \mathrm{dim}(K)$ and $\Gamma$ is cocompact in $G$.
\end{corollary}

As an immediate consequence, non-cocompact lattices don't appear in this setting and we have

\begin{corollary}\label{Nonexistence:general_lattice_actions}
Let $X$ be a geometric $3$-orbifold of finite volume, and $\Gamma$ a non-cocompact higher rank lattice in a semisimple Lie group $G$, then any action of $\Gamma$ in $X$ factors through a finite group.
\end{corollary}

As a particular example of the previous, any action of $SL_n(\Z)$ in a geometric $3$-orbifold of finite volume, factors through a finite group.

\begin{corollary}\label{Existence:general_lattice_actions}
Let $X$ be a geometric $3$-orbifold of finite volume, then $X$ admits an isometric action of a higher rank lattice $\Gamma \subset G$ if and only if the group $Iso(X)$ contains the group $SO(3)$. Moreover, the semisimple Lie group $G$ is isotypic of type $SO(3)$ and the lattice is uniform. 
\end{corollary}

Observe that the group $SO(4)$ factors locally as the product $SO(3) \times SO(3)$ and in fact, there is a copy of $SO(3)$ inside $SO(4)$, so that the previous Corollary includes at the same time examples like $X = S^3 / \Lambda$ and $X = (S^2 \times \R)/\Lambda$.

\bibliographystyle{alpha}
\bibliography{Rigidity}

\begin{thebibliography}{GGGNnZ20}

\bibitem[BBI01]{buragobook}
Dmitri Burago, Yuri Burago, and Sergei Ivanov.
\newblock {\em A course in metric geometry}, volume~33 of {\em Graduate Studies
  in Mathematics}.
\newblock American Mathematical Society, Providence, RI, 2001.

\bibitem[BFH16]{brownfisherhurtado2}
Aaron Brown, David Fisher, and Sebastian Hurtado.
\newblock Zimmer’s conjecture: subexponential growth, measure rigidity, and
  strong property (t).
\newblock {\em Preprint}, 2016.

\bibitem[BFH20]{brownfisherhurtado}
Aaron Brown, David Fisher, and Sebastian Hurtado.
\newblock Zimmer's conjecture for actions of {{\rm SL}(m, $\Bbb Z$)}.
\newblock {\em Invent. Math.}, 221(3):1001--1060, 2020.

\bibitem[BFH21]{brownfisherhurtado3}
Aaron Brown, David Fisher, and Sebastian Hurtado.
\newblock Zimmer's conjecture for non-uniform lattices and escape of mass,
  2021.

\bibitem[BK02]{bonkkleiner}
Mario Bonk and Bruce Kleiner.
\newblock Rigidity for quasi-{M}\"{o}bius group actions.
\newblock {\em J. Differential Geom.}, 61(1):81--106, 2002.

\bibitem[BK19]{bamlerkleiner}
Richard Bamler and Bruce Kleiner.
\newblock Ricci flow and contractibility of spaces of metrics.
\newblock arXiv:1909.08710, 2019.

\bibitem[BM46]{bochnerlocal}
Salomon Bochner and Deane Montgomery.
\newblock Locally compact groups of differentiable transformations.
\newblock {\em Ann. of Math. (2)}, 47:639--653, 1946.

\bibitem[BNnZ21]{barcenasnunezrmi}
No\'{e} B\'{a}rcenas and Jes\'{u}s N\'{u}\~{n}ez Zimbr\'{o}n.
\newblock On topological rigidity of {A}lexandrov 3-spaces.
\newblock {\em Rev. Mat. Iberoam.}, 37(5):1629--1639, 2021.

\bibitem[Boc46]{bochner}
S.~Bochner.
\newblock Vector fields and {R}icci curvature.
\newblock {\em Bull. Amer. Math. Soc.}, 52:776--797, 1946.

\bibitem[Bou09]{bourdon}
Marc Bourdon.
\newblock Quasi-conformal geometry and {M}ostow rigidity.
\newblock In {\em G\'{e}om\'{e}tries \`a courbure n\'{e}gative ou nulle,
  groupes discrets et rigidit\'{e}s}, volume~18 of {\em S\'{e}min. Congr.},
  pages 201--212. Soc. Math. France, Paris, 2009.

\bibitem[BRW61]{BREDON}
G.~E. Bredon, Frank Raymond, and R.~F. Williams.
\newblock {$p$}-adic groups of transformations.
\newblock {\em Trans. Amer. Math. Soc.}, 99:488--498, 1961.

\bibitem[BZ07]{bagaevzhukova}
A.~V. Bagaev and N.~I. Zhukova.
\newblock The isometry groups of {R}iemannian orbifolds.
\newblock {\em Sibirsk. Mat. Zh.}, 48(4):723--741, 2007.

\bibitem[DH20]{deroinhurtado}
Bertrand Deroin and Sebastian Hurtado.
\newblock Non left-orderability of lattices in higher rank semi-simple lie
  groups, 2020.
\newblock arXiv:2008.10687.

\bibitem[DVdW28]{dantzig}
D.~van Dantzig and B.~L. Van~der Waerden.
\newblock {\"U}ber metrisch homogene {R{\"a}ume}.
\newblock {\em Abh. Math. Semin. Univ. Hamb.}, 6:367--376, 1928.

\bibitem[Fis11]{fisher1}
David Fisher.
\newblock Groups acting on manifolds: around the {Z}immer program.
\newblock In {\em Geometry, rigidity, and group actions}, Chicago Lectures in
  Math., pages 72--157. Univ. Chicago Press, Chicago, IL, 2011.

\bibitem[Fis20]{fisher2}
David Fisher.
\newblock Recent developments in the {Z}immer program.
\newblock {\em Notices Amer. Math. Soc.}, 67(4):492--499, 2020.

\bibitem[FM98]{farbmasur}
Benson Farb and Howard Masur.
\newblock Superrigidity and mapping class groups.
\newblock {\em Topology}, 37(6):1169--1176, 1998.

\bibitem[FS00]{farbshalen}
Benson Farb and Peter Shalen.
\newblock Lattice actions, 3-manifolds and homology.
\newblock {\em Topology}, 39(3):573--587, 2000.

\bibitem[Fur76]{Fu}
Harry Furstenberg.
\newblock A note on {B}orel's density theorem.
\newblock {\em Proc. Amer. Math. Soc.}, 55(1):209--212, 1976.

\bibitem[FY94]{fukayayamaguchi}
Kenji Fukaya and Takao Yamaguchi.
\newblock Isometry groups of singular spaces.
\newblock {\em Math. Z.}, 216(1):31--44, 1994.

\bibitem[GG16]{galazglance}
Fernando Galaz-Garc\'{\i}a.
\newblock A glance at three-dimensional {A}lexandrov spaces.
\newblock {\em Front. Math. China}, 11(5):1189--1206, 2016.

\bibitem[GGG13]{galazguijarroisometry}
Fernando Galaz-Garcia and Luis Guijarro.
\newblock Isometry groups of {A}lexandrov spaces.
\newblock {\em Bull. Lond. Math. Soc.}, 45(3):567--579, 2013.

\bibitem[GGG15]{galazguijarrosurvey}
Fernando Galaz-Garcia and Luis Guijarro.
\newblock On three-dimensional {A}lexandrov spaces.
\newblock {\em Int. Math. Res. Not. IMRN}, (14):5560--5576, 2015.

\bibitem[GGGNnZ20]{galazguijarrochu}
Fernando Galaz-Garc\'ia, Luis Guijarro, and Jes\'us N\'u\~nez Zimbr\'on.
\newblock Sufficiently collapsed irreducible alexandrov 3-spaces are geometric.
\newblock {\em Indiana Univ. Math. J.}, 69(3):977--1005, 2020.

\bibitem[Gre74]{Greenberg}
Leon Greenberg.
\newblock Maximal groups and signatures.
\newblock In {\em Discontinuous groups and {R}iemann surfaces ({P}roc. {C}onf.,
  {U}niv. {M}aryland, {C}ollege {P}ark, {M}d., 1973)}, Ann. of Math. Studies,
  No. 79, pages 207--226. Princeton Univ. Press, Princeton, N.J., 1974.

\bibitem[Hae20]{Haettel}
Thomas Haettel.
\newblock Hyperbolic rigidity of higher rank lattices.
\newblock {\em Ann. Sci. \'{E}c. Norm. Sup\'{e}r. (4)}, 53(2):439--468, 2020.
\newblock With an appendix by Vincent Guirardel and Camille Horbez.

\bibitem[Hat83]{hatchersmale}
Allen~E. Hatcher.
\newblock A proof of the {S}male conjecture, {${\rm Diff}(S^{3})\simeq {\rm
  O}(4)$}.
\newblock {\em Ann. of Math. (2)}, 117(3):553--607, 1983.

\bibitem[Hem04]{hempel}
John Hempel.
\newblock {\em 3-manifolds}.
\newblock AMS Chelsea Publishing, Providence, RI, 2004.
\newblock Reprint of the 1976 original.

\bibitem[HJKL02]{hajokimlee}
Ku~Yong Ha, Jang~Hyun Jo, Seung~Won Kim, and Jong~Bum Lee.
\newblock Classification of free actions of finite groups on the 3-torus.
\newblock {\em Topology Appl.}, 121(3):469--507, 2002.

\bibitem[HKMR12]{diffeoelliptic}
Sungbok Hong, John Kalliongis, Darryl McCullough, and J.~Hyam Rubinstein.
\newblock {\em Diffeomorphisms of elliptic 3-manifolds}, volume 2055 of {\em
  Lecture Notes in Mathematics}.
\newblock Springer, Heidelberg, 2012.

\bibitem[HS17]{harvey-searle}
John Harvey and Catherine Searle.
\newblock Orientation and symmetries of {A}lexandrov spaces with applications
  in positive curvature.
\newblock {\em J. Geom. Anal.}, 27(2):1636--1666, 2017.

\bibitem[JL10]{jolee}
Jang~Hyun Jo and Jong~Bum Lee.
\newblock Group extensions and free actions by finite groups on solvmanifolds.
\newblock {\em Math. Nachr.}, 283(7):1054--1059, 2010.

\bibitem[KN96]{kobayashinomizu}
Shoshichi Kobayashi and Katsumi Nomizu.
\newblock {\em Foundations of differential geometry. {V}ol. {I}}.
\newblock Wiley Classics Library. John Wiley \& Sons, Inc., New York, 1996.
\newblock Reprint of the 1963 original, A Wiley-Interscience Publication.

\bibitem[Koj88]{kojima}
Sadayoshi Kojima.
\newblock Isometry transformations of hyperbolic {$3$}-manifolds.
\newblock {\em Topology Appl.}, 29(3):297--307, 1988.

\bibitem[KOS17]{kooohshin}
Daehwan Koo, Myungsung Oh, and Joonkook Shin.
\newblock Classification of free actions of finite groups on 3-dimensional
  nilmanifolds.
\newblock {\em J. Korean Math. Soc.}, 54(5):1411--1440, 2017.

\bibitem[Lee97]{LeeJoo}
Joo~Sung Lee.
\newblock Totally disconnected groups, {$p$}-adic groups and the
  {H}ilbert-{S}mith conjecture.
\newblock {\em Commun. Korean Math. Soc.}, 12(3):691--699, 1997.

\bibitem[LSY93]{leeshinyokura}
Kyung~Bai Lee, Joon~Kook Shin, and Shoji Yokura.
\newblock Free actions of finite abelian groups on the {$3$}-torus.
\newblock {\em Topology Appl.}, 53(2):153--175, 1993.

\bibitem[McC02]{mccullough}
Darryl McCullough.
\newblock Isometries of elliptic 3-manifolds.
\newblock {\em J. London Math. Soc. (2)}, 65(1):167--182, 2002.

\bibitem[Mor15]{morris}
Dave~Witte Morris.
\newblock {\em Introduction to arithmetic groups}.
\newblock Deductive Press, [place of publication not identified], 2015.

\bibitem[MS39]{myers-steenrod}
S.~B. Myers and N.~E. Steenrod.
\newblock The group of isometries of a {R}iemannian manifold.
\newblock {\em Ann. of Math. (2)}, 40(2):400--416, 1939.

\bibitem[MS19]{mecchiaseppi}
Mattia Mecchia and Andrea Seppi.
\newblock Isometry groups and mapping class groups of spherical 3-orbifolds.
\newblock {\em Math. Z.}, 292(3-4):1291--1314, 2019.

\bibitem[MZ74]{montgomeryzippin}
Deane Montgomery and Leo Zippin.
\newblock {\em Topological transformation groups.}
\newblock Robert E. Krieger Publishing Co., Huntington, N.Y.,,, 1974.
\newblock Reprint of the 1955 original.

\bibitem[Now34]{nowacki}
Werner Nowacki.
\newblock Die euklidischen, dreidimensionalen, geschlossenen und offenen
  {R}aumformen.
\newblock {\em Comment. Math. Helv.}, 7(1):81--93, 1934.

\bibitem[OS94]{OtsuShioya}
Yukio Otsu and Takashi Shioya.
\newblock The {R}iemannian structure of {A}lexandrov spaces.
\newblock {\em J. Differential Geom.}, 39(3):629--658, 1994.

\bibitem[Par13]{pardon}
John Pardon.
\newblock The {H}ilbert-{S}mith conjecture for three-manifolds.
\newblock {\em J. Amer. Math. Soc.}, 26(3):879--899, 2013.

\bibitem[Par19]{pardon1}
John Pardon.
\newblock Totally disconnected groups (not) acting on two-manifolds.
\newblock In {\em Breadth in contemporary topology}, volume 102 of {\em Proc.
  Sympos. Pure Math.}, pages 187--193. Amer. Math. Soc., Providence, RI, 2019.

\bibitem[Rag07]{Ra}
M.~S. Raghunathan.
\newblock Discrete subgroups of {L}ie groups.
\newblock {\em Math. Student}, (Special Centenary Volume):59--70 (2008), 2007.

\bibitem[Rat99]{ratcliffe99}
John~G. Ratcliffe.
\newblock On the isometry groups of hyperbolic orbifolds.
\newblock {\em Geom. Dedicata}, 78(1):63--67, 1999.

\bibitem[Rat19]{Rat}
John~G. Ratcliffe.
\newblock {\em Foundations of hyperbolic manifolds}, volume 149 of {\em
  Graduate Texts in Mathematics}.
\newblock Springer, Cham, 2019.
\newblock Third edition [of 1299730].

\bibitem[RS97]{RepScep}
Dusan Repovs and Evgenij Scepin.
\newblock A proof of the hilbert-smith conjecture for actions by lipschitz
  maps.
\newblock {\em Math. Ann.}, 308(2):361--364, 1997.

\bibitem[RT15]{ratcliffetschantz}
John~G. Ratcliffe and Steven~T. Tschantz.
\newblock On the isometry group of a compact flat orbifold.
\newblock {\em Geom. Dedicata}, 177:43--60, 2015.

\bibitem[Sco83]{Scott}
Peter Scott.
\newblock The geometries of {$3$}-manifolds.
\newblock {\em Bull. London Math. Soc.}, 15(5):401--487, 1983.

\bibitem[Tao14]{taofifth}
Terence Tao.
\newblock {\em Hilbert's fifth problem and related topics}, volume 153 of {\em
  Graduate Studies in Mathematics}.
\newblock American Mathematical Society, Providence, RI, 2014.

\bibitem[Thu97]{thu}
William~P. Thurston.
\newblock {\em Three-dimensional geometry and topology. {V}ol. 1}, volume~35 of
  {\em Princeton Mathematical Series}.
\newblock Princeton University Press, Princeton, NJ, 1997.
\newblock Edited by Silvio Levy.

\bibitem[Tol74]{tollefson}
Jeffrey~L. Tollefson.
\newblock The compact {$3$}-manifolds covered by {$S^{2}\times R^{1}$}.
\newblock {\em Proc. Amer. Math. Soc.}, 45:461--462, 1974.

\bibitem[UY17]{Yamada}
Masaaki Umehara and Kotaro Yamada.
\newblock {\em Differential geometry of curves and surfaces}.
\newblock World Scientific Publishing Co. Pte. Ltd., Hackensack, NJ, 2017.
\newblock Translated from the second (2015) Japanese edition by Wayne Rossman.

\bibitem[vN33]{vonneumann}
J.~von Neumann.
\newblock Die {E}inf\"{u}hrung analytischer {P}arameter in topologischen
  {G}ruppen.
\newblock {\em Ann. of Math. (2)}, 34(1):170--190, 1933.

\bibitem[Wal68]{waldhausenlarge}
Friedhelm Waldhausen.
\newblock On irreducible {$3$}-manifolds which are sufficiently large.
\newblock {\em Ann. of Math. (2)}, 87:56--88, 1968.

\bibitem[Wei11]{weinberger}
Shmuel Weinberger.
\newblock Some remarks inspired by the $c^0$ zimmer program.
\newblock In {\em Geometry, rigidity, and group actions}, Chicago Lectures in
  Math., pages 262--282. Univ. Chicago Press, Chicago, IL, 2011.

\bibitem[Yan60]{yang}
Chung-Tao Yang.
\newblock {$p$}-adic transformation groups.
\newblock {\em Michigan Math. J.}, 7:201--218, 1960.

\bibitem[Ye19]{ye2}
Shengkui Ye.
\newblock Symmetries of flat manifolds, jordan property and the general zimmer
  program.
\newblock {\em J. Lond. Math. Soc. (2)}, 100(3):1065--1080, 2019.

\bibitem[Ye20]{ye1}
Shengkui Ye.
\newblock A survey of topological's zimmer programm.
\newblock arXiv: arXiv:2002.01206, 2020.

\bibitem[Zim84]{zimmerbook}
Robert~J. Zimmer.
\newblock {\em Ergodic theory and semisimple groups}, volume~81 of {\em
  Monographs in Mathematics}.
\newblock Birkh\"{a}user Verlag, Basel, 1984.

\bibitem[Zim87]{zimmericm}
Robert~J. Zimmer.
\newblock Actions of semisimple groups and discrete subgroups.
\newblock In {\em Proceedings of the {I}nternational {C}ongress of
  {M}athematicians, {V}ol. 1, 2 ({B}erkeley, {C}alif., 1986)}, pages
  1247--1258. Amer. Math. Soc., Providence, RI, 1987.

\end{thebibliography}

\end{document}